\documentclass[a4paper,11pt]{article}

\usepackage[margin=2.5cm]{geometry}

\usepackage{bm}
\usepackage{amsmath}
\usepackage{amsfonts}
\usepackage{color}
\usepackage{amssymb}
\usepackage{graphicx}
\usepackage{enumerate}
\usepackage{amsthm,amscd}
\usepackage[all]{xy}
\usepackage{authblk}

\allowdisplaybreaks

\usepackage{hyperref}

\newtheorem{theorem}{Theorem}[section]
\newtheorem{construction}[theorem]{Construction}

\newtheorem{definition}[theorem]{Definition}

\newtheorem{lemma}[theorem]{Lemma}
\newtheorem{fact}[theorem]{Fact}
\newtheorem{proposition}[theorem]{Proposition}

\newtheorem{example}[theorem]{Example}
\newtheorem{remark}[theorem]{Remark}
\newtheorem{claim}[theorem]{Claim}

\begin{document}

\title{Complete Characterization on Maximum Pairwise Cross Intersecting Families (I)}
\author{
Yang Huang\thanks{School of Mathematics, Hunan University; Moscow Institute of Physics and Technology; 
E-mail: \url{1060393815@qq.com}.}
\,\,\,\,\,\,\,\,\,\,\,
Yuejian Peng\thanks{School of Mathematics, Hunan University; 
E-mail: \url{ypeng1@hnu.edu.cn}. 
The research is supported  by National Natural Science Foundation of China (No. 12571363) and National Natural Science Foundation of Hunan Province (Grant No. 2025JJ30003).}
}
\date{}



\maketitle

\vspace{-0.5cm}

\begin{abstract}
The families $\mathcal{A}$ and $\mathcal{B}$ are cross intersecting if $A\cap B\ne \emptyset$ for any $A\in \mathcal{A}$ and $B\in \mathcal{B}$.
Let $t\geq 2$ and $k_1\geq k_2\geq \cdots \geq k_t$.
We say that $(\mathcal{F}_1, \dots, \mathcal{F}_t)$ is an $(n, k_1, \dots, k_t)$-cross intersecting system if $\mathcal{F}_1 \subseteq{[n]\choose k_1}, \ldots ,\mathcal{F}_t \subseteq{[n]\choose k_t}$ are non-empty pairwise cross intersecting families.
Let $M(n,k_1,\ldots ,k_t)$ denote the maximum sum of sizes of families of an $(n,k_1,\ldots ,k_t)$-cross intersecting system.
The case $t=2$ was studied by Frankl and Tokushige (1992). Solving a problem of Shi, Frankl and Qian (2022), Huang, Peng and Wang (2024), as well as Zhang and Feng (2024) independently determined $M(n, k_1, \dots, k_t)$ for all $n\geq k_1+k_2$.

Observe that  $n\ge k_1 + k_t$ is the most  natural constraint.
However, the previous methods in the literature are invalid even in the range $n< k_1+k_2$.
In this paper, we overcome this obstacle and determine $M(n, k_1, \dots, k_t)$ for  $ k_1+k_3\le n < k_1+k_2$. Furthermore, we characterize all extremal structures. This could be viewed as the first `mixed-type' result about cross intersection problem.
Moreover, we introduce new concepts `$k$-partner',  `parity' and `corresponding $k$-set', and we develop some  methods to determine whether two L-initial cross intersecting families are maximal to each other.
In addition, we prove that in an extremal L-initial  $(n, k_1, \dots, k_t)$-cross intersecting system  $(\mathcal{F}_1,\ldots, \mathcal{F}_t)$, the sum $\sum_{i=1}^t{|\mathcal{F}_i|}$ can be expressed as a single variable function. We believe that our new result is of independent interest, and particularly, it plays an important role in the natural constraint $n\ge k_1+k_t$.
\end{abstract}

{{\bf Key words:}
Extremal combinatorics, Intersection theorems, Cross intersecting families}

{{\bf 2010 Mathematics Subject Classification.}  05D05, 05C65, 05D15.}

\section{Introduction}
Extremal combinatorics studies the maximum or minimum size of a combinatorial object that has certain properties.  One interesting class of extremal problems is related to the intersection problem, we refer the readers to the nice surveys \cite{FT16, Ellis2021}. This field is inspired by the fundamental  result of Erd\H{o}s, Ko and Rado \cite{EKR1961}.
We denote $[n]=\{1, 2, \dots, n\}$.  For $0\leq k \leq n$, let ${[n]\choose k}$ be the family of all $k$-element subsets of $[n]$.  A family  $\mathcal{A}$ of subsets of $[n]$ is  $k$-uniform if $\mathcal{A}\subseteq {[n]\choose k}$. A family $\mathcal{A}$ of sets is  {\em intersecting} if $A\cap B\ne \emptyset$ for any $A,B\in \mathcal{A}$. Half a century ago, Erd\H{o}s, Ko and Rado \cite{EKR1961} showed that  if $n> 2k$ and $\mathcal{A} \subseteq {[n] \choose k}$ is intersecting, then
$  |\mathcal{A}| \le {n-1 \choose k-1}$,
  with equality if and only if $\mathcal{A}$ consists of all $k$-subsets containing a fixed element.
Various intersection theorems were proved in the literature, and some of them were applied to other fields, e.g., the theoretical computer science and discrete geometry. For instance, the famous Ahlswede--Khachatrian theorem \cite{AK} that determines the maximum size of  a $k$-uniform family $\mathcal{A}$ where any two sets intersect in at least $t$ elements, which plays a crucial role in the hardness-of-approximation theorem for the `vertex cover' problem, proved by Dinur and Safra \cite{DS}.

Two families $\mathcal{A}$ and $\mathcal{B}$  are called {\em cross intersecting} if $A\cap B\ne \emptyset$ for any $A\in \mathcal{A}$ and $B\in \mathcal{B}$. In the case $\mathcal{A}=\mathcal{B}$, the problem for  cross intersecting families  reduces to that for intersecting families.
The families  $\mathcal{F}_1, \mathcal{F}_2,\dots, \mathcal{F}_t$  are called {\it non-empty pairwise cross intersecting} if every $\mathcal{F}_i$ is non-empty, and $\mathcal{F}_i$ and $\mathcal{F}_j$ are cross intersecting for any $ i\neq j $.
In 1967, Hilton and Milner  \cite{HM1967} determined the maximum sum of sizes of two cross intersecting families.

 \begin{theorem}[Hilton--Milner \cite{HM1967}]\label{HM}
Let $n$ and $k$ be positive integers with $n\geq 2k$.
 If $\mathcal{A}, \mathcal{B}\subseteq {[n]\choose k}$ are non-empty cross intersecting families, then
$$
|\mathcal{A}|+|\mathcal{B}|\leq {n\choose k} -{n-k\choose k}+1.
$$
\end{theorem}

Extending Theorem \ref{HM},
 Frankl and Tokushige \cite{FT} proved the following result.

 \begin{theorem}[Frankl--Tokushige \cite{FT}]\label{FT1992}
Let $\mathcal{A}\subseteq {[n]\choose k}$ and $\mathcal{B}\subseteq {[n]\choose \ell}$ be non-empty cross intersecting families with $n\geq k+ \ell$ and $k\geq \ell $. Then
$$|\mathcal{A}|+|\mathcal{B}|\leq{n\choose k}-{n-\ell \choose k}+1. $$
\end{theorem}

We refer the interested readers to \cite{BF2022, Fra2024, FT1998, FW2023, FW2024-EUJC, WFL2025} for more related results for cross intersecting families.
In 2020, Shi, Frankl and Qian \cite{SFQ2020} proved a weighted extension of Theorem \ref{FT1992} by bounding the maximum of $|\mathcal{A}| + c|\mathcal{B}|$ for a constant $c>0$.
 It is worth noting that an analogous weighted conclusion was independently obtained by Kupavskii \cite{Kupavskii2018}. Invoking the weighted version, Shi, Frankl and Qian \cite{SFQ2020} showed the following result for multi-families.

  \begin{theorem}[Shi--Frankl--Qian \cite{SFQ2020}] \label{SFQ}
Let $n,k$ and $t$ be positive integers with $n\ge 2k$ and $t\ge 2$.
If $\mathcal{F}_1, \mathcal{F}_2, \ldots ,\mathcal{F}_t \subseteq {[n] \choose k}$ are non-empty pairwise cross intersecting families, then
\[ \sum_{i=1}^t |\mathcal{F}_i|  \le
\max\left\{ {n \choose k} - {n-k \choose k} + t-1, t{n-1 \choose k-1}  \right\}. \]
\end{theorem}

For given integers $n, t$ and  $k_1\geq \dots\geq k_t$,
 we say that $(\mathcal{F}_1, \dots, \mathcal{F}_t)$ is an {\it $(n, k_1, \dots, k_t)$-cross intersecting system} if  $\mathcal{F}_1\subseteq {[n]\choose k_1}, \dots, \mathcal{F}_t\subseteq {[n]\choose k_t}$ are non-empty pairwise cross intersecting families.
Let $ M(n, k_1, \dots, k_t)$ be the maximum of $\sum_{i=1}^t |\mathcal{F}_i|$ over all $(n,k_1,\ldots ,k_t)$-cross intersecting systems.
Moreover, a system $(\mathcal{F}_1, \dots, \mathcal{F}_t)$  is {\em extremal} if $\sum_{i=1}^t|\mathcal{F}_i|=M(n, k_1, \dots, k_t)$.

A unified extension of Theorems \ref{FT1992} and \ref{SFQ} is to determine the value of $M(n, k_1,  \dots, k_t)$ for any integers $k_1\geq \dots \geq k_t$.
In the range $n\geq k_1+k_2$, the value of $M(n, k_1,  \dots, k_t)$ was determined  by Huang, Peng and Wang \cite{HPW},  and independently by Zhang and Feng \cite{ZF2023}.
\begin{theorem}[See \cite{HPW}, \cite{ZF2023}]\label{HP}
Let $\mathcal{F}_1\subseteq{[n]\choose k_1}, \mathcal{F}_2\subseteq{[n]\choose k_2}, \dots, \mathcal{F}_t\subseteq{[n]\choose k_t}$ be non-empty pairwise cross intersecting families with $t\geq 2$, $k_1\geq k_2\geq \cdots \geq k_t$. If $n\geq k_1+k_2$, then 
\begin{equation*}\label{value}
\sum_{i=1}^t{|\mathcal{F}_i|}\leq \textup{max} \left\{ 
\sum_{i=1}^t{n-1\choose k_i-1},\,\,  
{n\choose k_1}-{n-k_t\choose k_1}+\sum_{i=2}^t{{n-k_t\choose k_i-k_t}}\right\}.
\end{equation*}
Moreover, the extremal families are characterized as follows. 
\begin{itemize}
  \item[(i)] For $n>k_1+k_2$ or $n=k_1+k_2>k_1+k_t$, the equality holds if and only if
 either there is a set $T\in {[n] \choose k_t}$ such that $\mathcal{F}_1=\{F\in {[n]\choose k_1}\colon  F\cap T\ne\emptyset\}$ and $\mathcal{F}_j=\{F\in {[n]\choose k_j}: T\subseteq F\}$ for every $j\in [2,t]$; or  there exists $a\in [n]$ such that $\mathcal{F}_j=\left\{F\in{[n]\choose k_j}\colon  a\in F\right\}$ for every $j\in [t]$.

  \item[(ii)] For $t=2$ and $n=k_1+k_2$,  the equality holds if and only if $\mathcal{F}_1={[n]\choose k_1}\setminus \{[n]\setminus A\colon A\in \mathcal{F}_2\}$.

  \item[(iii)] For $t\geq 3$ and $n=k_1+k_2=k_1+k_t$, if $k_1>k_2$, then the equality holds if and only if $\mathcal{F}_1, \mathcal{F}_2,\ldots, \mathcal{F}_t$ are full stars with the same center; if $k_1=k_2$,  then the equality holds if and only if $\mathcal{F}_1=\mathcal{F}_2=\dots=\mathcal{F}_t$, and $\mathcal{F}_1$ is any intersecting family with size $\binom{n-1}{k_1-1}$.
\end{itemize}
\end{theorem}

This resolves an open problem of Shi, Frankl and Qian \cite[Problem 4.3]{SFQ2020}. More general,  a weighted version for multi-families was recently investigated by Huang and Peng \cite{HP+}.

Actually, the most natural constraint for this problem is $n\geq k_1+k_t$, instead of $n\ge k_1 + k_2$.  Indeed, if $n< k_1+k_t$, then for any extremal  $(n, k_1, \dots, k_t)$-cross intersecting system $(\mathcal{F}_1, \dots, \mathcal{F}_t)$, we must have $\mathcal{F}_1={[n]\choose k_1}$.
In this case, the problem of determining $M(n,k_1,k_2,\ldots ,k_t)$ reduces to that of determining $M(n, k_2, \dots, k_t)$.
However, when the constraint $n \ge k_1 + k_2$ is relaxed to $n \ge k_1 + k_t$, the methods previously established in \cite{HPW, ZF2023, HP+, SFQ2020} cease to be applicable.

For cross intersecting families $\mathcal{A}\subseteq {[n]\choose k}$ and
$\mathcal{B}\subseteq {[n]\choose \ell}$,
 we say that $\mathcal{A}$ and $\mathcal{B}$ are  {\it freely cross intersecting} whenever $n<k+\ell$; otherwise, we say that $\mathcal{A}$ and $\mathcal{B}$ are  {\it non-freely cross intersecting}.
 In addition, we say that the cross intersecting system $(\mathcal{F}_1,\ldots , \mathcal{F}_t)$  is of {\em mixed type} if it contains at least two families that are freely cross intersecting, and also contains at least two families that are non-freely cross intersecting; otherwise, we say that $(\mathcal{F}_1,\ldots , \mathcal{F}_t)$  is of {\em non-mixed type}.

The aforementioned references considered the intersection problem only for the non-mixed type systems, which require the condition $ n \ge k_1+k_2$.
It would be interesting  to study this problem for mixed type systems admitting the condition $n< k_1 + k_2$.
In this paper, we give the first mixed type result in the range $k_1+k_3\leq n <k_1+k_2$.
Given families $\mathcal{F}_1\subseteq {[n] \choose k_1}, \ldots , \mathcal{F}_t \subseteq {[n] \choose k_t}$ with
$k_1\geq k_2 \ge \cdots \geq k_t$, if $k_1+k_3\leq n <k_1+k_2$, then only
$\mathcal{F}_1$ and $\mathcal{F}_2$ are freely cross intersecting in the mixed type system $(\mathcal{F}_1,\ldots ,\mathcal{F}_t)$.
We begin with the following examples.

\begin{example}\label{con1}
Let $k_1\geq k_2\geq\dots\geq k_t$ and $k_1+k_3\leq n <k_1+k_2$. For each $i\in [t]$,  let
$$\mathcal{G}_i=\left\{ G\in {[n]\choose k_i}: 1\in G \right\}.$$
\end{example}

\begin{example}\label{con2}
Let $k_1\geq k_2\geq\dots\geq k_t$ and $k_1+k_3\leq n <k_1+k_2$.  For $i\in \{1,2\}$, let
$$\mathcal{H}_i=\left\{ H\in {[n]\choose k_i}: H\cap [k_t]\ne\emptyset \right\},$$
and for every $i\in[3, t]$, let
$$\mathcal{H}_i=\left\{ H\in {[n]\choose k_i}: [k_t]\subseteq H\right\}.$$
\end{example}

In the sequel, we shall show that under the constraint $k_1+k_3\leq n<k_1+k_2$, then an  $(n, k_1, \dots, k_t)$-cross intersecting system $(\mathcal{F}_1,\ldots ,\mathcal{F}_t)$ is extremal  if and only if it is isomorphic to the systems in Examples \ref{con1} or  \ref{con2}, with one more exception possibly.

\begin{theorem}[Main result]  \label{main1}
Let $t\geq 3$, $k_1\geq k_2\geq \cdots \geq k_t$ and $k_1+k_3\leq n<k_1+k_2$. If $\mathcal{F}_1\subseteq{[n]\choose k_1}, \mathcal{F}_2\subseteq{[n]\choose k_2}, \dots, \mathcal{F}_t\subseteq{[n]\choose k_t}$ are non-empty pairwise cross intersecting families,  then
$$
\sum_{i=1}^t{|\mathcal{F}_i|}\leq \textup{max} \left\{\sum_{i=1}^t{n-1\choose k_i-1}, \,\,\sum_{i=1}^2\left({n\choose k_i}-{n-k_t\choose k_i}\right)+\sum_{i=3}^t{{n-k_t\choose k_i-k_t}}\right\}.
$$
If $t=4$, $k_1=k_2$, $k_3=k_4$ and $n=k_1+k_3$, then the equality holds if and only if $\mathcal{F}_3=\mathcal{F}_4$ is an intersecting family, and $\mathcal{F}_1=\mathcal{F}_2= {[n]\choose k_1}\setminus \overline{\mathcal{F}_3}$. Otherwise,
the equality holds if and only if $(\mathcal{F}_1, \dots, \mathcal{F}_t)$ is isomorphic to $(\mathcal{G}_1, \dots, \mathcal{G}_t)$ in Example \ref{con1} or $(\mathcal{H}_1, \dots, \mathcal{H}_t)$ in Example \ref{con2}.
\end{theorem}

It is worth mentioning that the contribution of this paper not only lies in studying the case where $n<k_1+k_2$, but also provides a feasible approach toward ultimately resolving the most natural constraint $n\geq k_1+k_t$.
The key ingredient in our method is based on
utilizing the results involving the `$k$-partner', `parity' and `corresponding $k$-set', which are of independent interest.

\subsection{Sketch of our strategy}

To determine $M(n, k_1, \dots, k_t)$, a result of Kruskal and Katona (Theorem \ref{kk})  allows us to consider only families $\mathcal{F}_i$ whose members are the first $|\mathcal{F}_i|$  in lexicographic order (we call them  L-initial families) in an extremal $(n, k_1, \dots, k_t)$-cross intersecting system $(\mathcal{F}_1, \dots, \mathcal{F}_t)$.
   For an L-initial family $\mathcal{F}$, we call the last member (under lexicographic order) of $\mathcal{F}$ the {\em ID} of $\mathcal{F}$.
   In \cite{HP+}, we  came up with a new method different from \cite{HPW, ZF2023} and  bounded $\sum_{i=1}^t{c_i|\mathcal{F}_i|}$ (where $c_i$, $1\le i\le t$ are positive constants) by a function  $f_i(R)$, where $R$ is the ID of  $\mathcal{F}_i$,  and showed that $-f_i(R)$ has   `unimodality'. The method and results in our previous work \cite{HP+} provide some foundation for this paper, and we need new ingredients to overcome the difficulty raised by weakening the condition on $n$.  Since the condition on $n$ is relaxed to $k_1+k_t\leq n <k_1+k_2$, and at least two of $\mathcal{F}_1, \dots, \mathcal{F}_t$  are freely cross intersecting.
 When we try to  bound $\sum_{i=1}^t{|\mathcal{F}_i|}$ by a function, there are at least two free variables $I_1$ (the ID of $\mathcal{F}_1$) and $I_2$ (the ID of $\mathcal{F}_2$). This causes more difficulty to analyze properties of the corresponding function, comparing to the problem in \cite{HP+}. To overcome this difficulty,  we introduce new concepts `$k$-partner' , `parity' and `corresponding $k$-set', develop some rules to determine whether  a pair of L-initial cross intersecting families are maximal to each other, and
  prove one crucial property that
  if ($\mathcal{F}_1, \mathcal{F}_2, \dots, \mathcal{F}_t$) is  an extremal L-initial with IDs $I_1, I_2$ of $\mathcal{F}_1, \mathcal{F}_2$, respectively, then $I_2$  is the corresponding $k_2$-set of $I_1$. This discovery allows us to bound $\sum_{i=1}^t{|\mathcal{F}_i|}$ by a single variable function $f(I_1)$. Another crucial and challenge part is  to bound $f(I_1)$. Comparing to the function in  \cite{HP+}, we need to overcome more difficulties in dealing with  function $f(I_1)$ since there are more `mysterious' terms to be taken care of.
  We take advantage of some properties of function $f_i(R)$  obtained in \cite{HP+} and come up with some new strategies in estimating the change $f(I'_1)-f(I_1)$ as the ID of $\mathcal{F}_1$ increases from $I_1$ to $I'_1$.
The new concepts introduced in this paper—including the `$k$-partner', `parity', and `corresponding $k$-set'—along with the developed rules for determining whether a pair of $L$-initial cross intersecting families are maximal with respect to each other, will serve as a crucial foundation for the solution to the most general constraint $n\ge k_1+k_t$ . Key results, such as Lemma \ref{claim4},
are also essential to this framework.

\subsection{Organization of the paper}

In Section \ref{sec2}, we introduce the new concepts `$k$-partner' and `parity', and give some results  to determine whether  a pair of L-initial cross intersecting families are maximal to each other.
These results form a crucial foundation for addressing the general case where $n \ge k_1 + k_t$.
Section \ref{sec4} summarizes key results from our earlier work on the non-mixed type, which will be utilized in the present paper.
 In Section \ref{sec3}, we present the proof of Theorem \ref{main1}, assuming the validity of Lemma \ref{claim4}.
 Finally, in the last section, we provide the proof of Lemma \ref{claim4}, which will also play an essential role  in subsequent work \cite{hpgeneral}.

\section{Preliminaries}  \label{sec2}
In this section, we introduce new concepts `$k$-partner' and `parity', develop some rules to determine whether a pair of  L-initial cross intersecting families are maximal to each other.

When we write a set $A=\{a_1, a_2, \ldots, a_s\}\subseteq [n]$, we always assume that $a_1<a_2<\ldots<a_s$. We define the interval $[a, b]$ as follows:
 If $a< b$, then $[a, b]=\{a, a+1, \dots, b\}$.
If $a=b$, then $[a, b]=\{a\}=\{b\}$. If $a>b$, then $[a, b]=\emptyset$. For $A\subseteq [n]$, let $\max A$  and $\min A$ denote  the largest element and the smallest element of $A$, respectively.
 Let $A$ and $B$ be finite subsets of the set of positive integers. We say that $A$ is no more than $B$ in {\it lexicographic (lex for short) order}, denoted by $A\prec B$,
 if either $A\supseteq B$ or $\min(A\setminus B) < \min(B\setminus A)$. In particular, $A\prec A$. 
Let $\mathcal{L}([n], r, k)$ denote the first $r$ subsets in ${[n]\choose k}$ in the lex order. Given a set $R$, denote
\begin{equation}\label{NEW1}
\mathcal{L}([n], R, k)=\Big\{F\in {[n]\choose k}: F\prec R\Big\}.
\end{equation}
Whenever the underlying set $[n]$ is clear, we shall ignore it and write $\mathcal{L}(R, k)$, $\mathcal{L}(r, k)$ for short.  Let $\mathcal{F}\subseteq {[n]\choose k}$ be a family, we say $\mathcal{F}$ is {\it L-initial} if $\mathcal{F}=\mathcal{L}([n], |\mathcal{F}|, k)$. We call the last member of   $\mathcal{F}$ {\it ID} of $\mathcal{F}$.

The well-known Kruskal--Katona theorem \cite{KK1, KK2} will allow us to consider only L-initial families. (If we only consider the size of an $(n, k_1, \dots, k_t)$-cross intersecting system.) An equivalent formulation of this result was given in \cite{KK3, KK4} as follows.

\begin{theorem}[Kruskal--Katona theorem \cite{KK3, KK4}]\label{kk}
For $\mathcal{A}\subseteq {[n]\choose k}$ and $\mathcal{B}\subseteq {[n]\choose l}$, if $\mathcal{A}$ and $\mathcal{B}$ are cross intersecting, then $\mathcal{L}(|\mathcal{A}|, k)$ and $\mathcal{L}(|\mathcal{B}|, l)$ are cross intersecting as well.
\end{theorem}

Let $F$ and $H$ be two subsets of $[n]$ with size $|F|=f$ and $|H|=h$. We say that $F$ and $H$ {\it strongly intersect} at their last element if there is an element $q$ such that $F\cap H=\{q\}$ and $F\cup H=[q]$. We also say $F$ is the {\it partner} of $H$ (or $H$ is the {\it partner} of $F$).
In \cite{HP+}, the authors proved the following proposition.
\begin{proposition}\label{p2.7}
Let $a, b, n$ be positive integers and $a+b\leq n$. For $P\subseteq [n]$ with $|P|\leq a$, let $Q$ be the partner of $P$. Then $\mathcal{L}(Q, b)$ is the maximum  $b$-uniform family that is cross intersecting to $\mathcal{L}(P, a)$.
\end{proposition}

In \cite{HP+},  we worked on non-mixed type: Let $t\geq 2$, $k_1\geq k_2\geq \cdots \geq k_t$ and $n\geq k_1+k_2$  and families $\mathcal{A}_1\subseteq{[n]\choose k_1}, \mathcal{A}_2\subseteq{[n]\choose k_2}, \dots, \mathcal{A}_t\subseteq{[n]\choose k_t}$ be non-empty pairwise cross intersecting. Let $R$ be the ID of $\mathcal{A}_1$, and $T$  be the partner of $R$. In view of \cite{HP+}, one important ingredient is that by Proposition \ref{p2.7}, $\sum_{i=1}^t{|\mathcal{A}_i|}$ can be bounded by a function of $R$ as following.
\begin{align*}
f(R)=\sum_{j=1}^t|\mathcal{A}_j|\leq |\mathcal{L}(R, k_1)| +\sum_{j=2}^t |\mathcal{L}(T, k_j)| .
\end{align*}

Let $\mathcal{F}_1, \dots, \mathcal{F}_t$ be the families described in Theorem \ref{main1}.
By Theorem \ref{kk},  to prove the quantitative part of Theorem \ref{main1} we may  also assume that $\mathcal{F}_i$ is L-initial, that is, for each $i\in [t]$, we have $\mathcal{F}_i=\mathcal{L}(I_i, k_i)$, where $I_i$ is the ID of $\mathcal{F}_i$.
However,  the condition on $n$ is relaxed to $k_1+k_3\leq n <k_1+k_2$, so
 $\mathcal{F}_1\subseteq {[n]\choose k_1}$ and $\mathcal{F}_2\subseteq {[n]\choose k_2}$ are  freely cross intersecting. When we try to  bound $\sum_{i=1}^t{|\mathcal{F}_i|}$ by a function, there are two free variables $I_1$ (the ID of $\mathcal{F}_1$) and $I_2$ (the ID of $\mathcal{F}_2$). This causes more difficulty to analyze properties of the corresponding function, comparing to the problem in \cite{HP+}.  To overcome this difficulty,  some new ideas are needed.

   \subsection{Partner and Parity}

 Now, we introduce new concepts `$k$-partner' and `parity', and develop some criteria to determine whether  a pair of L-initial cross intersecting families are maximal to each other.

 To start with, we need to define the following concepts.

\begin{itemize}
\item {\bf $\mathcal{F}$ is maximal to $\mathcal{G}$.} Let $\mathcal{F}\subseteq {[n]\choose f}$ and  $\mathcal{G}\subseteq {[n]\choose g}$ be cross intersecting. We say that {\em $\mathcal{F}$ is maximal to $\mathcal{G}$} if $\mathcal{F}'\subseteq {[n]\choose f}$ and $\mathcal{G}$ are cross intersecting with $\mathcal{F}\subseteq \mathcal{F}'$, then $\mathcal{F}'= \mathcal{F}$.

\item {\bf $(\mathcal{F}, \mathcal{G})$ is maximal.}
The pair $(\mathcal{F}, \mathcal{G})$ is maximal if and only if $\mathcal{F}$ and $\mathcal{G}$ are maximal with respect to each other.

\item {\bf $F$ is maximal to $G$.} Let $F$ and $G$ be two subsets of $[n]$. We  say that {\em $F$ is maximal to $G$} if there are two L-initial families $\mathcal{F}\subseteq{[n]\choose |F|}$ and $\mathcal{G}\subseteq{[n]\choose |G|}$ with IDs $F$ and $G$ respectively such that $\mathcal{F}$ is maximal to $\mathcal{G}$.

\item {\bf $(F, G)$ is maximal.}
The pair $({F}, {G})$ is maximal if and only if ${F}$ and ${G}$ are maximal with respect to each other.

\item {\bf maximal pair families.} We say two families $\mathcal{A}_1$ and $\mathcal{A}_2$ are \emph{maximal pair families} if $|\mathcal{A}_1|=|\mathcal{A}_2|$ and for every $A_1\in \mathcal{A}_1$, there is a unique $A_2\in \mathcal{A}_2$ such that $(A_1, A_2)$ is maximal.

\item Let $F=\{x_1, x_2, \dots, x_k\}\subseteq [n]$. We denote
\begin{align*}
&\ell(F)=
\begin{cases}
\max \{x: [n-x+1, n]\subseteq F\}, & \text{if $\max F=n$};\\
0, & \text{if $\max F<n$}.
\end{cases}\\
&F^{\mathrm{t}}=[n-\ell(F)+1, n].
\end{align*}

\item {\bf $k$-partner.} Let $F$ and $H$ be two subsets of $[n]$ and partners to each other with  $|F|=f$ and $|H|=h$. Let $k\leq n-f$ be an integer, we define the {\it $k$-partner} $K$ of $F$ as follows. For $k=h$, let $K=H$.  If $k>h$, then let $K= H\cup \{n-k+h+1, \dots, n\}$.  If $k< h$, then let $K$ be the last $k$-set in ${[n]\choose k}$ such that $K \prec H$, in other words, there is no $k$-set $K'$ satisfying $K\precneqq K'\precneqq H$. We claim that $|K|=k$. To see this, we only need to check the case $k>h$. Indeed, since $F$ and $H$ intersect at their last element, $n'=\max H=f+h-1<n-k+h$. This implies $K^{\rm t}=\{n-k+h+1, \dots, n\}$ and $|K|=| H\cup \{n-k+h+1, \dots, n\}|=k$.

\item {\bf parity.}
Let $h_1$ and $h_2$ be positive integers with $h_1\leq h_2$, $H_1$ and $H_2$ be subsets of $[n]$ with sizes $h_1$ and $h_2$, respectively. We say $H_1$ is the {\it $h_1$-parity} of $H_2$ (or $H_2$ is the $h_2$-parity of $H_1$) if $H_1\setminus H_1^{\rm t}=H_2\setminus H_2^{\rm t}$ and $\ell(H_2)-\ell(H_1)=h_2-h_1$.
\end{itemize}

From the definition of $k$-partner, we have the following fact and remark.

\begin{fact}\label{f2.11}
Let $a, b, c, n$ be positive integers, $a\geq b$ and $n\geq a+c$, and let $C$ be a $c$-subset of $[n]$. Suppose that $A$ is the $a$-partner of $C$ and $B$ is the $b$-partner of $C$, then $B\prec A$ or $A$ is the $a$-parity of $B$.
\end{fact}

\begin{remark}\label{r2.6}
Let $F\subseteq [n]$ with $|F|=f$ and $k\leq n-f$. Suppose that  $H$ is the partner of $F$, and $K$ is the k-partner of $F$, then we have $\mathcal{L}(H, k)=\mathcal{L}(K, k)$.
\end{remark}

By Proposition \ref{p2.7} and Remark \ref{r2.6}, we have the following fact.
\begin{fact}\label{f2.8}
Let $a, b, n$ be positive integers and $n\geq a+b$. For  $A\subseteq [n]$ with $|A|=a$, let $B$ be the $b$-partner of $A$, then
 $\mathcal{L}(B, b)$ and $\mathcal{L}(A, a)$ are cross intersecting, moreover,
  $\mathcal{L}(B, b)$ is maximal to $\mathcal{L}(A, a)$.
\end{fact}

Note that  families $\mathcal{L}(A, a)$ and $\mathcal{L}(B, b)$ which mentioned in Fact \ref{f2.8} may not be maximal cross intersecting, since we don't know whether $\mathcal{L}(A, a)$ is maximal to $\mathcal{L}(B, b)$. For example, let $n=9$, $a=3$, $b=4$ and $A=\{2, 4, 7\}$. Then the $b$-partner of $A$ is $\{1, 3, 4, 9\}$. Although $\mathcal{L}(\{1, 3, 4, 9\}, 4)$ is maximal to $\mathcal{L}(\{2, 4, 7\}, 3)$, $\mathcal{L}(\{1, 3, 4, 9\}, 4)$ and $\mathcal{L}(\{2, 4, 7\}, 3)$ are not maximal cross intersecting families since $\mathcal{L}(\{2, 4, 7\}, 3)\subsetneq\mathcal{L}(\{2, 4, 9\}, 3)$, and $\mathcal{L}(\{2, 4, 9\}, 3)$ and $\mathcal{L}(\{1, 3, 4, 9\}, 4)$ are cross intersecting families.
\begin{fact}\label{7-21-2}
Let $a, b, n$ be positive integers and $n\geq a+b$.
Let $\mathcal{L}(A, a)$ and $\mathcal{L}(B, b)$ be cross intersecting families.
If $\{1, n-a+2, \dots, n\}\prec A$, then $B\prec \{1, n-b+2, \dots, n\}$.
\end{fact}

\begin{fact}\label{fact2.4+}
Let $F\subseteq [n]$ with $|F|=f$ and $k\leq n-f$. Suppose that $K$ and $K'$ are the $k$-partners of $F$ and $F\setminus F^{\rm t}$, respectively, then  $K=K'$.
\end{fact}

\begin{proof}
If $\ell(F)=0$, then we are done. Suppose $\ell(F)>0$ and $F=\{x_1, \dots, x_y\}\cup F^{\rm t}$.
Let $H$ and $H'$ be the partners of $F$ and $F\setminus F^{\rm t}$, respectively. Then $|H|=n-f+1>n-f$, consequently $k<|H|$,
$H=H\cap [x_y-1]\cup [x_y+1, n-\ell(F)]\cup \{n\}$ and
$H'=H\cap [x_y-1]\cup \{x_y\}$.
Recall that  $K$ and $K'$ are the $k$-partners of $F$ and $F\setminus F^{\rm t}$, respectively. By the definition of $k$-partner, we can see that
if $k\leq|H'|-1$,  then $K=K'$, as desired; if $k= |H'|$, then $K'=H'$ and $K=H\cap [x_y-1]\cup \{x_y\}=H'$, as desired; if $k>|H'|$, then $K'=H'\cup [n-k+|H'|+1, n]$ and $K=H\cap [x_y-1]\cup \{x_y\}\cup [n-k+|H'|+1, n]=K'$, as desired.
\end{proof}

Frankl and Kupavskii \cite{FK2018} gave a sufficient condition for a pair of maximal cross  intersecting families, and a necessary condition for a pair of maximal cross  intersecting families as stated below.
\begin{proposition}[Frankl--Kupavskii \cite{FK2018}]\label{FK+}
Let $a, b, n$ be positive integers with $n\geq a+b$. Let $P$ and $Q$ be non-empty subsets of $[n]$ with $|P|\leq a$ and $|Q|\leq b$. If $Q$ is the partner of $P$, then $\mathcal{L}(P, a)$ and $\mathcal{L}(Q, b)$ are maximal cross intersecting families.
Inversely, if $\mathcal{L}(A, a)$ and $\mathcal{L}(B, b)$ are maximal cross intersecting families with $|A|=a$ and $|B|=b$, let $j$ be the smallest element of $A\cap B$, $P=A\cap [j]$ and $Q=B\cap [j]$. Then $\mathcal{L}(P, a)=\mathcal{L}(A, a)$, $\mathcal{L}(Q, b)=\mathcal{L}(B, b)$ and $P$, $Q$ satisfy the following conditions:  $|P|\leq a$, $|Q|\leq b$, and $Q$ is the partner of $P$.
\end{proposition}

Based on Proposition \ref{FK+}, we point out a necessary and sufficient condition for a pair of maximal cross intersecting families in terms of their IDs.

\begin{lemma}\label{fact2.5}
Let $A$ and $B$ be non-empty subsets of $[n]$ with $|A|+|B|\leq n$. Let $A'=A\setminus A^{\mathrm{t}}$ and $B'=B\setminus B^{\mathrm{t}}$. Then $(A, B)$ is maximal if and only if $A'$ is the partner of $B'$.
\end{lemma}

\begin{proof}
Let $|A|=a$ and $|B|=b$.
Since $A'=A\setminus A^{\mathrm{t}}$ and $B'=B\setminus B^{\mathrm{t}}$, $\mathcal{L}(A, a)=\mathcal{L}(A', a)$ and $\mathcal{L}(B, b)=\mathcal{L}(B', b)$.
First we show the sufficiency.  Suppose that $A'$ is the partner of $B'$. Since $|A'|\leq |A|$ and $|B'|\leq |B|$, by Proposition \ref{FK+}, $\mathcal{L}(A', a)$ and $\mathcal{L}(B', b)$ are maximal cross intersecting families. Thus $\mathcal{L}(A, a)$ and $\mathcal{L}(B, b)$ are maximal cross intersecting families, in other words, $(A, B)$ is maximal.
Next, we show the necessity. Suppose that $(A, B)$ is maximal. Let $j$ be the smallest element of $A\cap B$, $P=A\cap [j]$ and $Q=B\cap [j]$. By Proposition \ref{FK+},
$\mathcal{L}(A, a)=\mathcal{L}(P, a)$, $\mathcal{L}(B, b)=\mathcal{L}(Q, b)$;
$|P|\leq a$, $|Q|\leq b$ and
 $P$ is the partner of $Q$.
Since $\mathcal{L}(A, a)=\mathcal{L}(P, a)$ and $P\subseteq A$,  $A=P\cup \{n-a+|P|+1, \dots, n\}$. Similarly,  $B=Q\cup \{n-b+|Q|+1, \dots, n\}$.
Note that $|P|\leq a$, $|Q|\leq b$ and $n\geq a+b$. Then $|Q|\leq n-a$ and $|P|\leq n-b$. Since $P$ is the partner of $Q$ and $\max P=\max Q=j$, $|P|+|Q|=j+1$.
So $j+1-|P|=|Q|\leq n-a$ and $j+1-|Q|=|P|\leq n-b$. Consequently, $j\leq n-a+|P|-1$ and $j\leq n-b+|Q|-1$. Hence, $P=A\setminus A^{\rm t}=A'$ and $Q=B\setminus B^{\rm t}=B'$.
 Since  $P$ is the partner of $Q$, $A'$ is the partner of $B'$.
\end{proof}

\begin{fact}\label{f2.10}
Let $a, b, n$ be positive integers and $n\geq a+b$. Suppose that $A$ is an $a$-subset of $[n]$, and $B$ is the $b$-partner of $A$. Let $A'$ be the  $a$-partner of $B$, then $(A', B)$ is maximal. Moreover, $A\prec A'$.
\end{fact}

\begin{proof}
By Fact \ref{f2.8},  $B$ is maximal to $A$. So $\mathcal{L}([n], A, a)$ and $\mathcal{L}([n], B, b)$ are cross intersecting. Since $A'$ is the  $a$-partner of $B$, by Fact \ref{f2.8} again,
$\mathcal{L}([n], A', a)$ and $\mathcal{L}([n], B, b)$ are cross intersecting and $A'$ is maximal to $B$. So $\mathcal{L}([n], A, a)\subseteq \mathcal{L}([n], A', a)$, i.e., $A\prec A'$.
\end{proof}

\begin{fact}\label{fact2.7}
Let $a, b, k, n$ be positive integers and $n\geq \max\{a+k, b+k\}$. Let $A$  and $B$ be two subsets of $[n]$ with sizes $|A|=a$ and $|B|=b$. Suppose that $K_a$ and $K_b$ are the $k$-partners of $A$ and $B$ respectively. If $A\prec B$, then $K_b\prec K_a$. In particular, if $A$ is the $a$-parity of $B$, then $K_a= K_b$.
\end{fact}

\begin{proof}
Let $K_A$ and $K_B$ be the partners of $A$ and $B$, respectively.
Since $A\prec B$, $K_B\prec K_A$. Therefore, by the definition of $k$-partner directly, $K_b\prec K_a$, as required.
For the second claim, since $A$ is the $a$-parity of $B$, $A\setminus A^{\rm t}=B\setminus B^{\rm t}$. Let $F$ be the $k$-partner of $A\setminus A^{\rm t}$. So $F$ is the $k$-partner of $B\setminus B^{\rm t}$ as well. By Fact \ref{fact2.4+}, $K_a=F$ and $K_b=F$. So $K_a=K_b$, as required.
\end{proof}

\begin{definition}\label{def2.7}
Let $h_1\leq h_2$. For two families $\mathcal{H}_1\subseteq {[n]\choose h_1}$ and $\mathcal{H}_2\subseteq {[n]\choose h_2}$, we say that $\mathcal{H}_1$ is the $h_1$-parity of $\mathcal{H}_2$ (or $\mathcal{H}_2$ is the $h_2$-parity of $\mathcal{H}_1$) if
the following properties hold:

(i) for any $H_1\in \mathcal{H}_1$, the $h_2$-parity of $H_1$ exists and must be in $\mathcal{H}_2$;

(ii) for any $H_2\in \mathcal{H}_2$, either $H_2$ has no $h_1$-parity or its $h_1$-parity is in $\mathcal{H}_1$.
\end{definition}

From the definition of parity,
 if a set  has a $k$-parity, then it has the unique one.
\begin{proposition}\label{prop2.9}
Let $f, g, h, n$ be positive integers with $f\geq g$ and $n\geq f+h$. Let
\begin{align*}
&\mathcal{F}=\left\{ F\in{[n]\choose f}: \text{ there exists } H\in {[n]\choose h} \text{ such that } (F, H) \text{ is maximal }  \right\},\\
&\mathcal{G}=\left\{ G\in{[n]\choose g}: \text{ there exists } H\in {[n]\choose h} \text{ such that } (G, H) \text{ is maximal }  \right\}.
\end{align*}
Let $\mathcal{H}_{\mathcal{F}}\subseteq {[n]\choose h}$ and $\mathcal{H}_{\mathcal{G}}\subseteq {[n]\choose h}$ be the families
such that $\mathcal{F}$ and $\mathcal{H}_{\mathcal{F}}$ are maximal pair families and $\mathcal{G}$ and $\mathcal{H}_{\mathcal{G}}$ are  maximal pair families.
Then the following properties hold:

(i) $\mathcal{F}$ is the $f$-parity of $\mathcal{G}$;

(ii) $\mathcal{H}_{\mathcal{G}}\subseteq \mathcal{H}_{\mathcal{F}}$;

(iii) for any $G\in \mathcal{G}$, let $F\in \mathcal{F}$ be the $f$-parity of $G$ and let $H\in \mathcal{H}$ be such that $(F, H)$ is maximal, then $(G, H)$ is maximal.
\end{proposition}

\begin{proof}
If $f=g$, then $\mathcal{F}=\mathcal{G}$ and $\mathcal{H}_{\mathcal{G}}=\mathcal{H}_{\mathcal{F}}$, we are done. We next consider the case that $f=g+s$ for some $s\geq 1$.

First, we are going to prove (i).
Let $G\in \mathcal{G}$. Then there is the unique $H\in \mathcal{H}$ such that $(G, H)$ is maximal. Let $G'=G\setminus G^{\rm t}=G\setminus \{n-\ell(G)+1, \dots, n\}$ and $H'=H\setminus H^{\rm t}=H\setminus \{n-\ell(H)+1, \dots, n\}$. Then, by Lemma \ref{fact2.5}, $G'$ and $H'$ are partners of each other.
So $\max G'=|G'|+|H'|-1=g-\ell(G)+|H'|-1=f-s-\ell(G)+|H'|-1\leq f+h-s-\ell(G)-1\leq n-\ell(G)-s-1$.
Let $F=G'\cup \{n-\ell(G)+1-s, \dots, n\}$. Then $\max G'\leq n-\ell(G)-s-1$ implies $F^{\rm t}=\{n-\ell(G)+1-s, \dots, n\}$.  Then $G\subseteq F$, $|F|=f$ and $\ell(F)-\ell(G)=f-g=s$. By the definition of parity, $F$ is the $f$-parity of $G$.  Let $F'=F\setminus F^{\rm t}$. So $F'=G'$. Furthermore, $F'$ and $H'$ are partners of each other. By Lemma \ref{fact2.5} again, $(F, H)$ is maximal. So $F\in \mathcal{F}$. Then to prove (i), we owe to confirm Definition \ref{def2.7} (ii). Let $F\in \mathcal{F}$. If $F$ dose not have $h_1$-parity, then we are done. Assume that $G$ is the $h_1$-parity of $F$. Then $G\setminus G^{\rm t}=F\setminus F^{\rm t}$. Since $F\in \mathcal{F}$, there exists $H$ in $\mathcal{H}_{\mathcal{F}}$ such that $(F, H)$ is maximal. By Lemma \ref{fact2.5}, $F\setminus F^{\rm t}$ is the partner of $H\setminus H^{\rm t}$. So $G\setminus G^{\rm t}$ is the partner of $H\setminus H^{\rm t}$ as well. By Lemma \ref{fact2.5} again, $(G, H)$ is maximal. So $G\in \mathcal{G}$. This implies that $\mathcal{F}$ is the $f$-parity of $\mathcal{G}$, as desired.

Next, we are going to prove (ii). Let $H\in \mathcal{H}_{\mathcal{G}}$. Then there exists $G\in \mathcal{G}$ such that $(G, H)$ is maximal. By (i), there exists $F\in \mathcal{F}$ such that $F$ is the $f$-parity of $G$.
By Lemma \ref{fact2.5}, $G\setminus G^{\rm t}$ is the partner of $H\setminus H^{\rm t}$. Since
$F$ is the $f$-parity of $G$, $G\setminus G^{\rm t}=F\setminus F^{\rm t}$. So
$F\setminus F^{\rm t}$ is the partner of $H\setminus H^{\rm t}$. By Lemma \ref{fact2.5} again, $(F, H)$ is maximal as well.
This implies $\mathcal{H}_{\mathcal{G}}\subseteq \mathcal{H}_{\mathcal{F}}$, as required.

At last, we are going to prove (iii).
Let $G\in \mathcal{G}$, and let $F$ be the $f$-parity of $G$. By (i), $F\in \mathcal{F}$. So there exists $H\in \mathcal{H}$  such that $(F, H)$ is maximal.
By Lemma \ref{fact2.5}, $F\setminus F^{\rm t}$ is the partner of $H\setminus H^{\rm t}$. Since
$F$ is the $f$-parity of $G$, $G\setminus G^{\rm t}=F\setminus F^{\rm t}$. So
$G\setminus G^{\rm t}$ is the partner of $H\setminus H^{\rm t}$. By Lemma \ref{fact2.5} again,
$(G, H)$ is maximal, as desired.
\end{proof}

\begin{proposition}\label{5-13-1}
Let $n=k+\ell$, $A$ be any $k$-subset of $[n]$. Then there is an $\ell$-set $B\subset [n]$ such that $(A, B)$ is maximal. Moreover, for any $A\in {[n]\choose k}$, let $B$ be the $\ell$-set such that $(A, B)$ is maximal, then we have $M(n, k, \ell)=|\mathcal{L}(A, k)|+|\mathcal{L}(B, \ell)|={n\choose \ell}$.
\end{proposition}

\begin{proof}
Let $A\in {[n]\choose k}$ and let $B'$ be the partner of $A\setminus A^{\mathrm{t}}$.  Since $n=k+\ell$, $|B'|\leq \ell$ and $\max B'<n$, in other words, $\ell(B')=0$.  If $|B'|=\ell$, then let $B=B'$. If $|B'|<\ell$, then let $B=B'\cup [n-\ell+|B'|+1, n]$. By Fact \ref{fact2.5}, $(A, B)$ is maximal. This implies
$\mathcal{L}(B, \ell)={[n]\choose \ell}\setminus \overline{\mathcal{L}(A, k)}$
since $n=k+\ell$. Thus, $M(n, k, \ell)=|\mathcal{L}(A, k)|+|\mathcal{L}(B, \ell)|={n\choose \ell}$, we are done.
\end{proof}

\section{Results of non-mixed type}\label{sec4}
In this section, we summarize some  results from our earlier work \cite{HP+} on the non-mixed type, which will be utilized in the present paper.
Denote
\begin{equation*}
\mathbb{R}_1=\Big\{ R\in {[n]\choose k_1}:\{1, n-k_1+2, \dots, n\}\prec R\prec \{k_{t}, n-k_1+2, \dots, n\}\Big\}.
\end{equation*}

  Two related functions are defined in \cite{HP+} as follows.

\begin{definition}\label{ab-}
Let $t\geq 2$, $k_1, k_2, \dots, k_{t}$ be positive integers with $k_1\geq k_2\geq \dots \geq k_{t}$ and $n\geq k_1+k_2$.
Let $R$ and $R'$  be of  $\mathbb{R}_1$ such that $R\prec R'$.
For each $j\in [2, t]$,
 let $T_j$ and $T'_j$ be the $k_j$-partners of $R$ and $R'$, respectively. We  define\footnote{Taking $i=1$, $d_1=\dots=d_t=1$ and $m_i=k_t$  in reference \cite[eq. (16) and (17)]{HP+}.}
\begin{align}\label{eq6}
&\alpha(R, R')=|\mathcal{L}(R', k_1)|-|\mathcal{L}(R, k_1)|,\\ \label{eq7}
&\beta(R, R')=\sum_{j=2}^t\big(|\mathcal{L}(T_j, k_j)|-|\mathcal{L}(T'_j, k_j)|\big).
\end{align}
\end{definition}
Notably,   the original definition in \cite{HP+}  regarding $\beta(R, R')$ is
 $\sum_{j=2}^t\big(|\mathcal{L}(T, k_j)|-|\mathcal{L}(T', k_j)|\big)$, where $T, T'$ are the partners of $R, R'$, respectively. By Remark \ref{r2.6}, these two definitions are equivalent.

\begin{definition}[$F_1<F_2$]\label{<}
 We say that $F_1<F_2$
 if $F_1\precneqq F_2$ and there is no $F'$ such that $|F'|=|F_1|$  and $F_1\precneqq F' \precneqq F_2$. 

\end{definition}



 The authors   proved the  following  result which constitutes an essential part of the `unimodality method' in  \cite{HP+}. (In fact,  \cite{HP+} presents a more general form.)

\begin{proposition}[Proposition 2.12 in \cite{HP+}]\label{clm23}
Let $n\geq k_1+k_2$ and $k_1\geq k_2\geq \dots\geq k_t$.
Let $F, G \in \mathbb{R}_1$ with $F<G$ and $\max G=q$. Then \\
(i) $\alpha(F, G)=1$.\\
(ii) $\beta(F, G)=\sum_{j=2}^t{n-q\choose k_j-(q-k_1)}$.\\
 (iii) If $n= k_1+k_j$ holds for any $j\in [t]\setminus \{1\}$, then $\beta(F, G)=t-1$; otherwise, $\beta(F, G)$ decreases when $q$ strictly increases until $\beta(F, G)=0$.
\end{proposition}

\begin{claim}[Claim 4.4 in \cite{HP+}]\footnote{In \cite{HP+}, the original statement is given by $\mathbb{R}_1(j)$ with $j\in [0, k_1-1]$; here, we take $j=0$. }\label{0000-}
Let $n\geq k_1+k_2$ and $k_1\geq k_2\geq \dots \geq k_t$.
Let $F_1, F'_1, G_1, G'_1 \in \mathbb{R}_1$ with $F_1<F'_1, G_1<G'_1$ and  $\max F'_1=\max G'_1$. Then $\alpha(F_1, F'_1)=\alpha(G_1, G'_1)$ and $\beta(F_1, F'_1)=\beta(G_1, G'_1)$.
\end{claim}

\begin{definition}[{\bf $c$-sequential}]\label{def3.14}
Let $\mathcal{A}\subseteq {[n]\choose k}$ be a family and $c\in [k]$. We say that $\mathcal{A}$ is $c$-{\it sequential} if there are $A\subseteq [n]$ with $|A|=k-c$ and $a\geq \max A$ such that $\mathcal{A}=\{A\sqcup \{a+1, \dots, a+c\}, A\sqcup \{a+2, \dots, a+c+1\}, \dots, A\sqcup \{b-c+1, \dots, b\}\}$,  write $A_1\overset{c}{\prec} A_2\overset{c}{\prec}\cdots\overset{c}{\prec}A_{b-a-c+1}$, where $A_1=A\sqcup \{a+1, \dots, a+c\}$, $A_2=A\sqcup \{a+2, \dots, a+c+1\}$,$\dots$, $A_{b-a-c+1}=A\sqcup \{b-c+1, \dots, b\}$.
 For any $A_i$,  $A_j$ contained in $\mathcal{A}$, we also say $A_i$, $ A_j$ are $c$-sequential,
 in particular, if $j=i+1$, then we write $A_{i}\overset{c}{\prec}  A_{j}$. If  $\max A_{j}=n$ for the last $j$, then we write $A_{i}\overset{c}{\longrightarrow}  A_{j}$ for any $i<j$.
 \end{definition}

\begin{lemma}[Claim 4.3 in \cite{HP+}]\footnote{In \cite{HP+}, the original statement is given by $\mathbb{R}_1(j)$ with $j\in [0, k_1-1]$; here, we take $j=0$. }\label{clm27-}
Let $n\geq k_1+k_2$ and $k_1\geq k_2\geq \dots \geq k_t$.
Let 
$A_1, B_1, C_1, D_1\in \mathbb{R}_1$.
Suppose that $A_1, B_1$ are $c$-sequential and  $C_1, D_1$ are $c$-sequential with $\max A_1=\max C_1$ and $\max B_1=\max D_1$. Then $\alpha(A_1, B_1)=\alpha(C_1, D_1)$ and
$\beta(A_1, B_1)=\beta(C_1, D_1)$.
\end{lemma}

\begin{proposition}[Proposition 3.1 in \cite{HP+}]\label{7-17-2}
Suppose that $n\geq k_1+k_2 $, $k_1\geq k_2\geq \dots \geq k_t$ are positive integers. Let $m_i=\min \{k_j: j\in [t]\setminus \{i\}\}$.
For each $s\in [m_i]$ and each $i\in [t]$, let
\[
f_i(\{s\})={n\choose k_i}-{n-s\choose k_i}+\sum_{j\ne i}{n-s\choose k_j-x}.
\]
Then
\begin{align*}
&f_1(\{s\})=\max \{f_j(\{s\}): j\in [t]\} \text{ for each } s\in [k_t],\\
&f_t(\{k_{t-1}\})\leq \max \{f_1(\{1\}), f_1(\{k_t\})\}.
\end{align*}
\end{proposition}

\begin{theorem}[Theorem 1.9   in \cite{HP+}]  \label{thm1.9}
Let $n$, $t\geq 2$, $k_1, k_2, \dots, k_t$ be positive integers and $d_1, d_2, \dots, d_t$ be positive numbers. Let
$\mathcal{A}_1\subseteq{[n]\choose k_1}, \mathcal{A}_2\subseteq{[n]\choose k_2}, \dots, \mathcal{A}_t\subseteq{[n]\choose k_t}$ be non-empty pairwise cross-intersecting families with $|\mathcal{A}_i|\geq {n-1\choose k_i-1}$ for some $i\in [t]$. Let $m_i$ be the minimum integer among $k_j$, where $j\in [t]\setminus \{i\}$. If $n\geq k_i+k_j$ for all $j\in [t]\setminus \{i\}$, then
\begin{equation}\label{new1}
\sum_{j=1}^td_j|\mathcal{A}_j|\leq \max  \left\{d_i{n\choose k_i}-d_i{n-m_i\choose k_i}+\sum_{j\ne i}d_j{n-m_i\choose k_j-m_i}, \,\,\sum_{j=1}^td_j{n-1\choose k_j-1}\right\},
\end{equation}
the equality holds if and only if one of the following holds.\\
(i)  $d_i{n\choose k_i}-d_i{n-m_i \choose k_i}+\sum_{j\ne i}d_j{n-m_i\choose k_j-m_i}\geq \sum_{j=1}^td_j{n-1\choose k_j-1}$, and  there is some $m_i$-element set $T\subseteq [n]$ such that $\mathcal{A}_i=\{F\in {[n]\choose k_i}: F\cap T\ne\emptyset\}$ and $\mathcal{A}_j=\{F\in {[n]\choose k_j}: T\subseteq F\}$ for each $j\in [t]\setminus \{i\}$.\\
(ii) $d_i{n\choose k_i}-d_i{n-m_i \choose k_i}+\sum_{j\ne i}d_j{n-m_i\choose k_j-m_i}\leq \sum_{j=1}^td_j{n-1\choose k_j-1}$, and
 there is some $a\in [n]$ such that $\mathcal{A}_j=\{F\in{[n]\choose k_j}: a\in F\}$ for each $j\in [t]$. \\
(iii) $n=k_i+k_j$ holds for every $j\in [t]\setminus \{i\}$ and $\sum_{j\ne i}d_j>d_i$. Let $k_j=k$ for all $j\ne i$.
If $t=2$, then $\mathcal{A}_{3-i}\subseteq {[n]\choose k_{3-i}}$ with $|\mathcal{A}_{3-i}|={n-1\choose k_{3-i}-1}$ and $\mathcal{A}_i={[n]\choose k_i}\setminus \overline{\mathcal{A}_{3-i}}$.  If $t\geq3$, then for each $j\in [t]\setminus \{i\}$, $\mathcal{A}_j=\mathcal{A}$ and $\mathcal{A}_i={[n]\choose k_i}\setminus \overline{\mathcal{A}}$, where $\mathcal{A}$ is a $k$-uniform intersecting family with size $|\mathcal{A}|= {n-1\choose k-1}$.\\
(iv) $n=k_i+k_j$ holds for every $j\in [t]\setminus \{i\}$ and $\sum_{j\ne i}d_j=d_i$. Let $k_j=k$ for all $j\ne i$.
If $t=2$, then $\mathcal{A}_{3-i}\subseteq {[n]\choose k_{3-i}}$ with $1\leq|\mathcal{A}_{3-i}|\leq{n-1\choose k_{3-i}-1}$ and $\mathcal{A}_i={[n]\choose k_i}\setminus \overline{\mathcal{A}_{3-i}}$.
If $t\geq3$, then for each $j\in [t]\setminus \{i\}$, $\mathcal{A}_j=\mathcal{A}$ and $\mathcal{A}_i={[n]\choose k_i}\setminus \overline{\mathcal{A}}$, where $\mathcal{A}$ is a $k$-uniform intersecting family with size $1\leq |\mathcal{A}|\leq {n-1\choose k-1}$.
\end{theorem}

\section{Proof of Theorem \ref{main1}}\label{sec3}
Denote
\begin{align*}
&\lambda_1:=\sum_{i=1}^t{n-1\choose k_i-1},\\
&\lambda_2:=\sum_{i=1}^2\left({n\choose k_i}-{n-k_t\choose k_i}\right)+\sum_{i=3}^t{{n-k_t\choose k_i-k_t}}.
\end{align*}

Recall Examples \ref{con1} and \ref{con2}. Clearly, in Example \ref{con1},
$\sum_{i=1}^t|\mathcal{G}_i|=\lambda_1$ and Example \ref{con2},
$\sum_{i=1}^t|\mathcal{H}_i|=\lambda_2$. Thus,
we have the following remark.

\begin{remark}\label{NEW2}
Let $k_1\geq k_2\geq\dots\geq k_t$ and $k_1+k_3\leq n <k_1+k_2$.
If $(\mathcal{F}_1, \dots, \mathcal{F}_t)$ is an  extremal $(n, k_1, \dots, k_t)$-cross intersecting system, then
\begin{equation*}
\sum_{i=1}^t|\mathcal{F}_i|\geq \textup{max}\{\lambda_1, \,\, \lambda_2\}.
\end{equation*}
\end{remark}

The following remark derives from Fact \ref{7-21-2}
\begin{remark}\label{remark2.20+}
Let $d\geq 2$ be an integer. We consider $d$ L-initial families $\mathcal{L}([n], A_1, a_1)$, $\dots$, $\mathcal{L}([n], A_d, a_d)$ with $|A_1|=a_1, \dots, |A_d|=a_d$. For some fixed $i\in [d]$, let $S\subseteq [d]\setminus \{i\}$ be the set of all $j\in [d]\setminus \{i\}$ satisfying that $n\geq a_i+a_j$. Suppose that $\{1, n-a_i+2, \dots, n\}\preceq  A_i$. If $\mathcal{L}(A_i, a_i)$ and $\mathcal{L}(A_j, a_j)$ are cross intersecting for each $j\in S$, then
$\mathcal{L}(A_j, a_j)$, $j\in S$ are pairwise cross intersecting families since for any $j\in S$, every member of $\mathcal{L}(A_j, a_j)$ contains $1$.
\end{remark}

\begin{lemma}\label{>}
Let $k_1\geq k_2\geq\dots\geq k_t\geq 2$ and $k_1+k_3\leq n <k_1+k_2$.
Suppose that $(\mathcal{F}_1, \dots, \mathcal{F}_t)$ is an L-initial extremal $(n, k_1, \dots, k_t)$-cross intersecting system with IDs $I_1, I_2, \dots, I_t$  of $\mathcal{F}_1, \mathcal{F}_2, \dots, \mathcal{F}_t$, respectively. Then
 $|\mathcal{F}_1|\geq {n-1\choose k_1-1}$ and $|\mathcal{F}_2|\geq {n-1\choose k_2-1}$, equivalently, $\{1, n-k_1+2, \dots, n\} \prec I_1$ and $\{1, n-k_2+2, \dots, n\} \prec I_2$.
\end{lemma}

\begin{proof}
  If  $|\mathcal{F}_i|< {n-1\choose k_i-1}$ for all $i\in [t]$, then $\sum_{i=1}^t{|\mathcal{F}_i|}< \lambda_1$, a contradiction to Remark \ref{NEW2}. So there is $i\in [t]$ such that $|\mathcal{F}_i|\geq {n-1\choose k_i-1}$.

First, we consider the case: $i\in [3, t]$ and  $I_j\precneqq \{1, n-k_j+2, \dots, n\}$ for each $j\in [2]$.
Since  $k_1\geq k_2\geq\dots\geq k_t\geq 2$ and $n\geq k_1+k_3$,   $n\geq k_i+k_j$ for all $j\in [t]\setminus \{i\}$.
Let $m_i=\min \{k_j: j\in [t]\setminus \{i\}\}$.
Taking $d_1=d_2=\dots=d_t=1$ in Theorem \ref{thm1.9}, we obtain
\[
\sum_{j=1}^t|\mathcal{F}_j|
 = \textup{max}\left\{ \sum_{j=1}^t{n-1\choose k_j-1}, \,\,{n\choose k_i}-{n-m_i \choose k_i}+\sum_{j\in [t]\setminus \{i\}}{n-m_i\choose k_j-m_i}  \right\}.
 \]
 Note that $n>k_i+m_i$ and  for each $j\in [2]$, $I_j\precneqq \{1, n-k_j+2, \dots, n\}$. We only meet Theorem \ref{thm1.9}(i). Thus,
\[
\sum_{j=1}^t|\mathcal{F}_j|
= {n\choose k_i}-{n-m_i \choose k_i}+\sum_{j\in [t]\setminus \{i\}}{n-m_i\choose k_j-m_i}.
\]
 Setting $s=m_i$ in Proposition \ref{7-17-2}, we have
\begin{equation*}
{n\choose k_i} \!- \! {n\!-\!m_i \choose k_i}+ \! \! \sum_{j\not\in \{2, i\}} \!\!\!
{n\!-\!m_i\choose k_j \!-\!m_i}\leq
\max\Bigg\{
{n\choose k_1} \!-\! {n\!-\!k_t\choose k_1}
\!+ \!
\sum_{j=3}^t{n\!-\!k_t\choose k_j\!-\!k_t}, \,
\sum_{j\ne2}{n\!-\!1\choose k_j \!-\! 1}
\Bigg\}.
\end{equation*}
Since $m_i\geq k_t\geq 2$,
$
{n-m_i\choose k_2-m_i}<{n-1\choose k_2-1}$ and
${n-m_i\choose k_2-m_i}<{n\choose k_2}-{n-k_t\choose k_2}$.
Thus,
 $
 \sum_{j=1}^t|\mathcal{F}_j|<\max \{\lambda_1, \lambda_2\},
$
 a contradiction to Remark \ref{NEW2}.

Thus $i\in [2]$.
 Without loss of generality, we may assume that $i=1$. Since $n<k_1+k_2$, then  any two families  $\mathcal{G}\subseteq {[n]\choose k_1}$ and  $\mathcal{H}\subseteq{[n]\choose k_2}$ are freely cross intersecting. Note that  for each $j\in [3, t]$, $n\geq k_1+k_j$ and $\mathcal{F}_1$ is cross intersecting with $\mathcal{F}_j$.
 Since $\{1, n-k_1+2, \dots, n\}\prec I_1$, by Remark \ref{remark2.20+},  every member of $\mathcal{F}_j$ contains 1. Since $(\mathcal{F}_1, \dots, \mathcal{F}_t)$ is extremal,
 all $k_2$-subsets containing 1 are contained in $\mathcal{F}_2$, so  $|\mathcal{F}_2|\geq {n-1\choose k_2-1}$, as required.
\end{proof}


\begin{definition}\label{R_2}
For a set $A\subset [n]$, we define the {\it corresponding $k$-set} $B\subset [n]$ of $A$ as follows:  If $A$ has the $k$-parity, then let $B$ be the $k$-parity of $A$; otherwise, let $B$ be the $k$-set such that $B< A$.
\end{definition}

Let $k<\ell$, 
for two sets $P, R$, if $P<R$ with sizes $|P|=|R|=\ell$,  by case analysis on whether $P, R$ have
$k$-parity, we have the following fact.


\begin{fact}\label{correspondingk}
Let $k<\ell$ be positive integers, $P, R\in {[n]\choose \ell}$ and $\{1, n-l+2, \dots, n\} \precneqq  P<R$. Then there exist the  corresponding $k$-sets for both $P$ and $R$, denoted by $P'$ and $R'$ respectively. Moreover, $P'\prec R'$, and $P'=R'$ if and only if $R$ does not have  the $k$-parity.
\end{fact}

\begin{proposition}\label{7-18-1}
Let $k_1\geq k_2\geq\dots\geq k_t\geq 2$ and $k_1+k_3\leq n <k_1+k_2$.
Suppose that $(\mathcal{F}_1, \dots, \mathcal{F}_t)$ is an L-initial extremal $(n, k_1, \dots, k_t)$-cross intersecting system with  IDs $I_1, I_2, \dots, I_t$  of $\mathcal{F}_1, \mathcal{F}_2, \dots, \mathcal{F}_t$, respectively. Then the following properties holds:
\begin{itemize}
\item[(i)] $I_2$ is the corresponding $k_2$-set of $I_1$;
\item[(ii)] for each $i\in [3, t]$, $I_i$ is the $k_i$-partner of $I_1$.
\end{itemize}
\end{proposition}

\begin{proof}
By Lemma \ref{>}, we have $\{1, n-k_1+2, \dots, n\} \prec I_1$ and $\{1, n-k_2+2, \dots, n\} \prec I_2$.
For each $i\in [3, t]$,  let $P_i, Q_i$ be the $k_i$-partners of $I_1, I_2$, respectively.

First, we consider the case: $I_2\prec I_1$.
Let $i\in [3, t]$.
Since $n\geq k_1+k_i\geq k_2+k_i$ and $I_2\prec I_1$,
 by  Fact \ref{fact2.7}, $P_i\prec Q_i$. Note that $\mathcal{F}_1$ and $\mathcal{F}_i$ are cross intersecting and  $n\geq k_1+k_i$. By Fact \ref{f2.8}, $P_i$ is maximal to $I_1$. Therefore, $I_i\prec P_i$.
 By Remark \ref{remark2.20+}, $(\mathcal{F}_1, \mathcal{F}_2, \mathcal{L}(P_3, k_3), \dots, \mathcal{L}(P_t, k_t))$ is a cross intersecting system.
 Since
 $(\mathcal{F}_1, \dots, \mathcal{F}_t)$ is extremal, $I_i= P_1$, i.e., $I_i$ is $k_i$-partner of $I_1$, we get (ii). Therefore, $I_2$ is lexicographically maximal without exceeding $I_1$. Thus, $I_2$ is the corresponding $k_2$-set of $I_1$, we get (i).

Next, we consider the case: $I_1\prec I_2$.
Let $i\in [3, t]$.
Using a similar argument as above, we obtain that $I_i$ is $k_i$-partner of $I_2$.
By Fact \ref{fact2.7}, to prove Proposition \ref{7-18-1}, it suffices to prove that $I_1$ is the $k_1$-parity of $I_2$. To see this, we need the following claim.
\begin{claim}\label{7-18-2}
$(I_2, I_3)$ is maximal.
\end{claim}

\begin{proof}
Suppose on the contrary that  $(I_2, I_3)$ is not maximal.
Let $I'_2$ be the $k_2$-partner of $I_3$.
Note that $I_3$ is the $k_3$-partner of $I_2$.
Since $(I_2, I_3)$ is not maximal, in view of Fact \ref{f2.10},
$I_2\precneqq I'_2$ and $(I'_2, I_3)$ is maximal. So $I_2$ and $I'_2$ share the  same $k_3$-partner $I_3$. Consequently, they also share the  same $k_i$-partner $I_i$ for each $i\in [3, t]$.
Thus $(\mathcal{F}_1, \mathcal{L}(I'_2, k_2), \mathcal{F}_3, \dots, \mathcal{F}_t)$ is cross intersecting. However, $\mathcal{F}_2 \subsetneqq \mathcal{L}(I'_2, k_2)$, a contradiction to the fact that $(\mathcal{F}_1, \dots, \mathcal{F}_t)$ is extremal. This completes the proof of Claim \ref{7-18-2}.
\end{proof}

Let us continue the proof of Proposition \ref{7-18-1}.
Since $n\geq k_1+k_3$ and $k_1\geq k_2$, combining Proposition \ref{prop2.9}
and Claim \ref{7-18-2}, $I_2$ has $k_1$-parity. Since $I_1\prec I_2$ and $(\mathcal{F}_1, \dots, \mathcal{F}_t)$ is extremal, $I_1$ is the $k_1$-parity of $I_2$. This completes the proof of Proposition \ref{7-18-1}.
\end{proof}

For every $i\in [2]$, we denote
\begin{align*}
\mathcal{R}_i=\Bigg\{R\in {[n]\choose k_i}:  \{1, n-k_i+2, \dots, n\} \prec R \prec \{k_t, n-k_i+2, \dots, n\}\Bigg\}.
\end{align*}

 For every $i\in[3, t]$, we denote
 \begin{align*}
\mathcal{R}_i&=\Bigg\{ R\in {[n]\choose k_i}: [k_t]\cup[ n-k_i+k_t+1, n] \prec R \prec \{1, n-k_i+2, \dots, n\}\Bigg\}.
\end{align*}

As a consequence of Lemma \ref{>} and Proposition \ref{7-18-1}, we have the following remark.

\begin{remark}\label{NEW4}
Let $k_1\geq k_2\geq\dots\geq k_t\geq 2$ and $k_1+k_3\leq n <k_1+k_2$. If $(\mathcal{F}_1, \dots, \mathcal{F}_t)$ is an L-initial extremal $(n, k_1, \dots, k_t)$-cross intersecting system with IDs $I_1, I_2, \dots, I_t$ of $\mathcal{F}_1, \mathcal{F}_2, \dots, \mathcal{F}_t$, respectively, then  for each $i\in [t]$,  we have $I_i \in \mathcal{R}_i$.
\end{remark}

\begin{proposition}\label{prop2.4}
Let $k_1\geq k_2\geq\dots\geq k_t\geq 2$ and $k_1+k_3\leq n <k_1+k_2$, and
let $(\mathcal{F}_1, \dots, \mathcal{F}_t)$ be an L-initial extremal $(n, k_1, \dots, k_t)$-cross intersecting system with with IDs $I_1, I_2, \dots, I_t$ of $\mathcal{F}_1, \mathcal{F}_2, \dots, \mathcal{F}_t$, respectively.
If the set $I_1$ is either $\{1, n-k_1+2, \dots, n\}$ or $\{k_t, n-k_1+2, \dots, n\}$, then $\sum_{i=1}^t|\mathcal{F}_i|= \textup{max}\{\lambda_1, \,\, \lambda_2\}$.
\end{proposition}

\begin{proof}
By Proposition \ref{7-18-1} (i), if $I_1=\{1, n-k_1+2, \dots, n\}$, then $I_2=\{1, n-k_2+2, \dots, n\}$; if $I_1=\{k_t, n-k_1+2, \dots, n\}$, then $I_2=\{k_t, n-k_2+2, \dots, n\}$.

 Suppose first that $I_1=\{1, n-k_1+2, \dots, n\}$.
 By Proposition \ref{7-18-1} (ii), for each $i\in [3, t]$, $I_i$ is the $k_i$-partner of $I_1$, therefore, $I_i=\{1, n-k_i+2, \dots, n\}$.
Thus $\sum_{i=1}^t|\mathcal{F}_i|=\lambda_1$.

Next, suppose that  $I_1=\{k_t, n-k_1+2, \dots, n\}$.  By Proposition \ref{7-18-1} (ii) again, for each $i\in [3, t]$, $I_i=\{1, \dots, k_t, n-k_i+k_t+1, \dots, n\}$.
 Thus $\sum_{i=1}^t|\mathcal{F}_i|=\lambda_2$.
\end{proof}

Let $(\mathcal{F}_1, \dots, \mathcal{F}_t)$ be an L-initial extremal $(n, k_1, \dots, k_t)$-cross intersecting system with $k_1\geq k_2\geq\dots\geq k_t\geq 2$ and $k_1+k_3\leq n <k_1+k_2$.
 Remark \ref{NEW4} tells us all  possible IDs of $\mathcal{F}_i$ for each $i\in [t]$.
Proposition \ref{prop2.4} shows that if the ID of $\mathcal{F}_1$ is the first or the last member of $\mathcal{R}_1$, then $\sum_{i=1}^t|\mathcal{F}_i|$ equals $\lambda_1$ or $\lambda_2$. A natural question is whether $\sum_{i=1}^t|\mathcal{F}_i|$ is larger than $\lambda_1$ and $\lambda_2$ if the ID of $\mathcal{F}_1$ is in the intermediate region of $\mathcal{R}_1$.
The following theorem shows that the answer is negative, except for an  exceptional  case.

\begin{theorem}\label{AS+}
Suppose that $t\geq 4$, $k_1\geq k_2\geq\dots\geq k_t\geq 2$ and $k_1+k_3\leq n <k_1+k_2$. Additionly, suppose that if $t= 4$ with $k_1=k_2$ and $k_3=k_4$, then $n>k_1+k_3$.
Let $(\mathcal{F}_1, \dots, \mathcal{F}_t)$ be an L-initial $(n, k_1, \dots, k_t)$-cross intersecting system with ID $I_1$ of $\mathcal{F}_1$.
If $\{1, n-k_1+2, \dots, n\}\precneqq   I_1\precneqq \{k_t, n-k_1+2, \dots, n\}$, then $(\mathcal{F}_1, \dots, \mathcal{F}_t)$ is not extremal.
\end{theorem}

To prove Theorem \ref{AS+}, our idea is to find an $(n, k_1, \dots, k_t)$-cross intersecting system $(\mathcal{L}(K_1, k_1)$, $\dots, \mathcal{L}(K_t, k_t))$ with $K_i \in \mathcal{R}_i$ for each $i\in [t]$ such that $\sum_{i=1}^t|\mathcal{L}(K_i, k_i)|>\sum_{i=1}^t|\mathcal{F}_i|$.
For this purpose, we are going to show `unimodality'  of the forthcoming function $f(R_1)$, where $R_1\in \mathcal{R}_1$. For this purpose, we need to make some preparations first.

\begin{definition}\label{ab}
Let $s, s', t$ be integers with $s\geq s'\geq 1$, $t\geq s+1$, $k_1=\dots=k_{s'}> k_{s'+1}\geq\dots\geq k_t$, and $k_1+k_{s+1}\leq n <k_{s-1}+k_s$.
Suppose that $R_1, R'_1\in \mathcal{R}_1$ with $R_1\prec R'_1$. For each $i\in [s+1, t]$, let $R_i, R'_i$ be the $k_i$-partners of $R_1, R'_1$, respectively; for each $i\in [2, s]$, let $R_i, R'_i$ be the corresponding $k_i$-sets of $R_1, R'_1$, respectively.
For every $i\in [s]$, we denote
\begin{align*}
\alpha_i(R_1, R'_1)&=|\mathcal{L}(R'_i, k_i)|-|\mathcal{L}(R_i, k_i)|,\\
\gamma(R_1, R'_1)&=\sum_{i=1}^s\alpha_i(R_1, R'_1),\\
\delta(R_1, R'_1)&=\sum_{i=s+1}^t\big(|\mathcal{L}(R_i, k_i)|-|\mathcal{L}(R'_i, k_i)|\big).
\end{align*}
\end{definition}

Comparing two Definitions \ref{ab} and \ref{ab-}, we have the following remark.
\begin{remark}\label{remark4.14}
Let $s, s', t$ be integers with $s\geq s'\geq 1$, $t\geq s+1$, $k_1=\dots=k_{s'}> k_{s'+1}\geq\dots\geq k_t$ and $k_1+k_{s+1}\leq n <k_{s-1}+k_s$.
Let $R_1, R_1' \in \mathcal{R}_1$ with $R_1\prec R'_1$.
Replacing $k_1, k_2, \dots, k_t$ by $k_1, k_{s+1}, \dots, k_t$ in Definition \ref{ab-},
  we can see  that $\delta(R_1, R'_1)=\beta(R_1, R'_1)$, and  for each $i\in [s']$,
$
\alpha_i(R_1, R'_1)=\alpha(R_1, R'_1)$.
\end{remark}

\begin{definition}\label{deff}
Let $s, s', t$ be integers with $s\geq s'\geq 1$, $t\geq s+1$, $k_1=\dots=k_{s'}> k_{s'+1}\geq\dots\geq k_t$, and $k_1+k_{s+1}\leq n <k_{s-1}+k_s$.
Let $R_1\in \mathcal{R}_1$. For each $i\in [2, s]$, let $R_i$ be  the corresponding $k_i$-set of $R_1$, and for each $i\in [s+1, t]$, let $R_i$ be the $k_i$-partner of $R_1$. Denote
$$f(R_1)=\sum_{i=1}^t |\mathcal{L}(R_i, k_t)|.$$
\end{definition}

From Definitions \ref{ab} and \ref{add}, for two sets $R_1, R'_1\in \mathcal{R}_1$ with $R_1\prec R'_1$, we have
\begin{equation}\label{7-19-4}
f(R'_1)-f(R_1)=\gamma(R_1, R'_1)-\delta(R_1, R'_1).
\end{equation}

\begin{remark}\label{7-19-1}
Let $s, s', t$ be integers with $s\geq s'\geq 1$, $t\geq s+1$, $k_1=\dots=k_{s'}> k_{s'+1}\geq\dots\geq k_t$, and $k_1+k_{s+1}\leq n <k_{s-1}+k_s$.
Let $R_1\in \mathcal{R}_1$. For each $i\in [2, s]$, let $R_i$ be  the corresponding $k_i$-set of $R_1$, and for each $i\in [s+1, t]$, let $R_i$ be the $k_i$-partner of $R_1$.
Then $(\mathcal{L}(R_1, k_1), \dots, \mathcal{L}(R_t, k_t))$ is an $(n, k_1, \dots, k_t)$-cross intersecting system.
\end{remark}

\begin{proof}
Since $n <k_{s-1}+k_s$, $\mathcal{L}(R_{1}, k_{1}), \dots, \mathcal{L}(R_s, k_s)$ are pairwise cross intersecting.
Let $i\in [s]$,  $j\in [s+1, t]$, and let $R_{i, j}$ be the $k_j$-partner of $R_i$.
By Fact \ref{f2.8}, $\mathcal{L}(R_i, k_i)$ and $\mathcal{L}(R_{i, j}, k_j)$ are cross intersecting.
Note that $R_j$ is the $k_j$-partner of $R_1$, and $R_i\prec R_1$.
By Fact \ref{fact2.7}, $R_j \prec R_{i, j}$.
Thus $\mathcal{L}(R_i, k_i)$ and  $\mathcal{L}(R_j, k_j)$ are cross intersecting.
Since $R_1\in \mathcal{R}_1$ and $R_i$ is  the corresponding $k_i$-set of $R_1$, $\{1, n-k_i+2, \dots, n\}\prec R_i\prec R_1$. Thus, it follows from Remark \ref{remark2.20+}, that $\mathcal{L}(R_{s+1}, k_{s+1}), \dots, \mathcal{L}(R_t, k_t)$ are pairwise cross intersecting. Therefore, $(\mathcal{L}(R_1, k_1), \dots, \mathcal{L}(R_t, k_t))$ is an $(n, k_1, \dots, k_t)$-cross intersecting system.
\end{proof}

Combining  Proposition \ref{7-18-1} and Remark \ref{7-19-1}, and taking $s=2$ in Definition \ref{deff},  we can see that if
$(\mathcal{F}_1, \dots, \mathcal{F}_t)$ is an extremal L-initial $(n, k_1, \dots, k_t)$-cross intersecting with ID $I_1$ of $\mathcal{F}_1$,  then
$$\sum_{i=1}^t|\mathcal{F}_i|=f(I_1)=\max \{f(R_1): R_1\in \mathcal{R}_1\}.$$
As a consequence, to prove Theorem \ref{AS+}, it suffices to prove the following theorem.

\begin{theorem}\label{7-19-2}
Let $s, s', t$ be integers with $s\geq s'\geq 1$, $t\geq s+1$, $k_1=\dots=k_{s'}> k_{s'+1}\geq\dots\geq k_t$, and $k_1+k_{s+1}\leq n <k_{s-1}+k_s$.  Additionally, suppose that if $n=k_1+k_t$, then $s'\ne t-s$.
Let $R_1\in \mathcal{R}_1$. If $\{1, n-k_1+2, \dots, n\}\precneqq   R_1\precneqq \{k_t, n-k_1+2, \dots, n\}$,  then there is $K_1\in \mathcal{R}_1$ such that $f(K_1)>f(R_1)$.
\end{theorem}

To prove the above theorem, we need  the following lemma, whose proof will be given in the next section.
We emphasize that this lemma will be crucial in our follow-up paper \cite{hpgeneral} for handing the most general case where $n\geq k_1+k_t$.

\begin{lemma}\label{claim4}
Let $s, s', t$ be integers with $s\geq s'\geq 1$, $t\geq s+1$, $k_1=\dots=k_{s'}> k_{s'+1}\geq\dots\geq k_t$, and $k_1+k_{s+1}\leq n <k_{s-1}+k_s$.
Let $A_1, B_1, C_1\in \mathcal{R}_1$ with  $A_1\setminus A_1^{\rm t}=A\sqcup \{a, a+1\}$, $B_1=A\sqcup \{a\}\sqcup [n-\ell(B_1)+1, n]$ and $C_1=A\sqcup \{a+1\}\sqcup [n-\ell(B_1)+1, n]$ (where $\max A<a$).
Then $\delta(A_1, B_1)=\delta(B_1, C_1)$ and $\gamma(A_1, B_1)\leq \gamma(B_1, C_1)$.
\end{lemma}

We are going to prove Theorem \ref{7-19-2} by applying Proposition \ref{clm23} and Lemma \ref{claim4}.

\begin{proof}[Proof of Theorem \ref{7-19-2}]
Suppose on the contrary that
\begin{equation}\label{7-19-3}
f(R_1)=\max \{f(F_1): F_1\in \mathcal{R}_1\}.
\end{equation}
We may assume that $R_1$ is the first member of $\mathcal{R}_1$ satisfying (\ref{7-19-3}) and $\{1, n-k_1+2, \dots, n\}\precneqq   R_1\precneqq \{k_t, n-k_1+2, \dots, n\}$. We have the following two cases.

{\bf Case 1.} $\ell(R_1)>0$. In this case, $R_1^t=[n-\ell(R_1)+1, n]$.
Denote
\begin{align*}
&P_1=R_1\setminus \{n-\ell(R_1)+1\}\cup \{\max (R_1\setminus R_1^t) +1\}\\
&Q_1=R_1\setminus \{\max (R_1\setminus R_1^t)\}\cup \{\max (R_1\setminus R_1^t) +1\}.
\end{align*}
Since $\{1, n-k_1+2, \dots, n\}\precneqq R_1\precneqq \{k_t, n-k_1+2, \dots, n\}$, $P_1, Q_1$ are contained in $\mathcal{R}_1$, and $\{1, n-k_1+2, \dots, n\}\precneqq P_1 \precneqq \{k_t, n-k_1+2, \dots, n\}$. By the choice of $R_1$, we have
$
f(P_1)<f(R_1).
$
In view of (\ref{7-19-4}), $f(R_1)=f(P_1)+\gamma(P_1, R_1)-\delta(P_1, R_1)$. Thus,
\begin{equation}\label{7-19-5}
\gamma(P_1, R_1)>\delta(P_1, R_1)
\end{equation}
Note that $P_1, R_1, Q_1$ satisfy the conditions of Lemma \ref{claim4},  corresponding to $A_1, B_1, C_1$, respectively.
B Lemma \ref{claim4}, we have
$\delta(P_1, R_1)=\delta(R_1, Q_1)$ and $\gamma(P_1, R_1)\leq \gamma(R_1, Q_1)$. Combining these with (\ref{7-19-5}), and in view of (\ref{7-19-4}), we have
\[
f(Q_1)=f(R_1)+\gamma(R_1, Q_1)-\delta(R_1, Q_1)\geq f(R_1)+\gamma(P_1, R_1)-\delta(P_1, R_1)>f(R_1),
\]
a contradiction to (\ref{7-19-3}).

{\bf Case 2.} $\ell(R_1)=0$. Let $P'_1, Q'_1\in \mathcal{R}_1$ be such that $P'_1<R_1<Q'_1$. (Recall that `$<$' is defined in Definition \ref{<}.) Since $\{1, n-k_1+2, \dots, n\}\precneqq R_1\precneqq \{k_t, n-k_1+2, \dots, n\}$, $P'_1, Q'_1$ are contained in $\mathcal{R}_1$. Thus $f(P'_1)\leq f(R_1)$. In view of (\ref{7-19-4}), $f(R_1)=f(P'_1)+\gamma(P'_1, R_1)-\delta(P'_1, R_1)$. Thus,
\begin{equation}\label{7-19-7}
\gamma(P'_1, R_1)\geq \delta(P'_1, R_1)
\end{equation}
For each $i\in [2, s]$, let $P'_i, R_i, Q'_i$ be the corresponding $k_i$-sets of $P'_1, R_1, Q'_1$, respectively;
for each $i\in [s+1, t]$, let $P'_i, R_i, Q'_i$ be the $k_i$-partners of $P'_1, R_1, Q'_1$, respectively.

Since  $\ell(R_1)=0$ and $k_i<k_1$ for $i\in [s'+1, s]$,  by Fact \ref{correspondingk}, we have the following claim. 
\begin{claim}\label{7-19-8}
For each $i\in [s'+1, s]$, we have $P'_i=R_i\prec Q'_i$.
\end{claim}

Note that $P'_1<R_1<Q'_1$.
Thus, Claim \ref{7-19-8} gives
\begin{equation}\label{7-20-2}
\gamma(P'_1, R_1)=s'\leq \gamma(R_1, Q'_1).
\end{equation}
Since $\max R_1<\max Q'_1$ and $P'_1<R_1<Q'_1$, combining Proposition \ref{clm23} and Remark \ref{remark4.14}, we obtain one of the following three cases:
$\delta(P'_1, R_1)>\delta(R_1, Q'_1)$;  $\delta(P'_1, R_1)=\delta(R_1, Q'_1)=0$;
 $\delta(P'_1, R_1)=\delta(R_1, Q'_1)=t-s$, where the last case happens if and only if $n=k_1+k_t$.
So $\delta(P'_1, R_1)\geq \delta(R_1, Q'_1)$.
If $\gamma(P'_1, R_1)> \delta(P'_1, R_1)$ in (\ref{7-19-7}), then by (\ref{7-20-2}), we have
$$f(Q'_1)=f(R_1)+\gamma(R_1, Q'_1)-\delta(R_1, Q'_1)\geq f(R_1)+\gamma(P'_1, R_1)-\delta(P'_1, R_1)>f(R_1),$$
a contradiction to (\ref{7-19-3}).
Thus $\gamma(P'_1, R_1)= \delta(P'_1, R_1)$ in (\ref{7-19-7}).
If $\delta(P'_1, R_1)> \delta(R_1, Q'_1)$ or $\gamma(P'_1, R_1)< \gamma(R_1, Q'_1)$ in (\ref{7-20-2}),
then
$$f(Q'_1)=f(R_1)+\gamma(R_1, Q'_1)-\delta(R_1, Q'_1)> f(R_1)+\gamma(P'_1, R_1)-\delta(P'_1, R_1)\geq f(R_1),$$
a contradiction to (\ref{7-19-3}). Thus
$$t-s=\delta(R_1, Q'_1)=\delta(P'_1, R_1)=\gamma(P'_1, R_1)= \gamma(R_1, Q'_1)=s',$$
and $n=k_1+k_t$.
This makes a contradiction to the condition of Theorem \ref{7-19-2}.

The proof of Theorem \ref{7-19-2} is complete.
\end{proof}

Thus, we complete the proof of Theorem \ref{AS+}.
We are ready to prove Theorem \ref{main1}.

\begin{proof}[Proof of Theorem \ref{main1}]
Let $(\mathcal{F}_1, \dots, \mathcal{F}_t)$ be an extremal $(n, k_1, \dots, k_t)$-cross intersecting system with $k_1\geq k_2\geq\dots\geq k_t$ and $k_1+k_3\leq n <k_1+k_2$.

We  first consider the case: $k_t=1$. Suppose that $|\mathcal{F}_t|=s$. Then $\mathcal{F}_t=\{\{1\}, \dots, \{s\}\}$. Since $(\mathcal{F}_1, \dots, \mathcal{F}_t)$ is a cross intersecting system, then  for any $i\in [t-1]$ and any $F\in \mathcal{F}_i$, we have $[s]\subseteq F$. Thus,
$$\sum_{i=1}^t|\mathcal{F}_i|=\sum_{i=1}^{t-1}{n-s\choose k_i-s}+s\leq\lambda_1,$$
as required.
Note that the above equality holds if and only if $s=1$, that is $\mathcal{F}_t=\{x\}$ for some $x\in [n]$. Therefore, $(\mathcal{F}_1, \dots, \mathcal{F}_t)$ is isomorphic to $(\mathcal{G}_1, \dots, \mathcal{G}_t)$ which is defined in Example \ref{con1}, we are done.

Next, we consider the case: $t=3$. By Theorem \ref{FT1992}, $|\mathcal{F}_1|+|\mathcal{F}_3|\leq {n\choose k_1}-{n-k_3\choose k_1}+1$. Clearly, $|\mathcal{F}_2|\leq {n\choose k_2}-{n-k_3\choose k_2}$ since every $k_2$-set in $\mathcal{F}_2$ intersects with every $k_3$-set in $\mathcal{F}_3$. Thus, $\sum_{i=1}^3|\mathcal{F}_i|\leq \lambda_2$, as required. The upper bound can be achieved if and only if $|\mathcal{F}_2|= {n\choose k_2}-{n-k_3\choose k_2}$. This implies that there is some $k_3$-set $A$, such that $\mathcal{F}_3=\{A\}$ and $\mathcal{F}_2=\{F\in {[n]\choose k_2}: F\cap A\ne \emptyset\}$. Therefore, $\mathcal{F}_1=\{F\in {[n]\choose k_1}: F\cap A\ne \emptyset\}$. So $(\mathcal{F}_1, \mathcal{F}_3, \mathcal{F}_3)$ is isomorphic to $(\mathcal{H}_1, \mathcal{H}_2, \mathcal{H}_3)$ which is defined in Example \ref{con2}, we are done.

Then, we  consider the case: $t=4$, $k_1=k_2$, $k_3=k_4$ and $n=k_1+k_3$. In this case,
$|\mathcal{F}_1|+|\mathcal{F}_4|\leq {n\choose k_1}$ and $|\mathcal{F}_2|+|\mathcal{F}_3|\leq {n\choose k_1}$.
Thus $\sum_{i=1}^4|\mathcal{F}_i|\leq 2{n\choose k_1}=\lambda_1=\lambda_2$, as required.
Since $(\mathcal{F}_1, \dots, \mathcal{F}_4)$
is extremal, $\mathcal{F}_1=\mathcal{F}_2$ and $\mathcal{F}_3=\mathcal{F}_4$. Note that $n>k_3+k_4$ and $n=k_1+k_3$. Then $\mathcal{F}_3$ is an intersecting family, and $\mathcal{F}_1= {[n]\choose k_1}\setminus \overline{\mathcal{F}_3}$, we are done.

Now, it suffices to  prove Theorem \ref{main1} under the following conditions.
\begin{equation}\label{7-17-1}
\text{$k_t\geq 2$, $t\geq 4$ and if $t=4$ with $k_1=k_2$ and $k_3=k_4$, then $n>k_1+k_3$.}
\end{equation}
The the quantitative part  immediately follows from  Theorem \ref{kk}, Proposition \ref{prop2.4} and Theorem \ref{AS+}.
We are going to show that  $(\mathcal{F}_1, \dots, \mathcal{F}_t)$  is isomorphic to Examples \ref{con1} or \ref{con2} under conditions in (\ref{7-17-1}).

By Lemma \ref{>}, Proposition \ref{prop2.4}, Theorem \ref{AS+}, and Theorem \ref{kk}, we conclude that:  either for each $i\in [t]$, we have $|\mathcal{F}_i|={n-1\choose k_i-1}$ or $|\mathcal{F}_1|={n\choose k_1}-{n-k_t\choose k_1}$, $|\mathcal{F}_2|={n\choose k_2}-{n-k_t\choose k_2}$ and $|\mathcal{F}_i|={n-k_t\choose k_i-k_t}$ holds for each $i\in [3, t]$.
If the first case happens, then $(\mathcal{F}_1, \mathcal{F}_3,\dots, \mathcal{F}_t)$ is an $(n, k_1, k_3, \dots, k_t)$-cross intersecting system with
$\sum_{i=1, i\ne 2}^t|\mathcal{F}_i|=\sum_{i=1, i\ne 2}^t{n-1\choose k_i-1}$.
Note that $k_1>k_3$ and $t\geq 4$. By Theorem \ref{HP} (i),  $(\mathcal{F}_1, \dots, \mathcal{F}_t)$ is isomorphic to $(\mathcal{G}_1, \dots, \mathcal{G}_t)$ which is defined in Example \ref{con1}.
If the second case happens, then $(\mathcal{F}_1, \mathcal{F}_3,\dots, \mathcal{F}_t)$ is an $(n, k_1, k_3, \dots,  k_t)$-cross intersecting system with
$\sum_{i=1, i\ne 2}^t|\mathcal{F}_i|={n\choose k_1}-{n-k_t\choose k_1}+\sum_{i=3}^t{n-k_t\choose k_i-k_t}$
By Theorem \ref{HP}  (i),  $(\mathcal{F}_1, \dots, \mathcal{F}_t)$ is isomorphic to $(\mathcal{H}_1, \dots, \mathcal{H}_t)$ which is defined in Example \ref{con2}.
 This completes the proof of Theorem \ref{main1}.
 \end{proof}

\section{Proof of Lemma \ref{claim4}}\label{sec5}
From Definition \ref{ab}, we have the following remark.
\begin{remark}\label{add}
Let $s, s', t$ be integers with $s\geq s'\geq 1$, $t\geq s+1$, $k_1=\dots=k_{s'}> k_{s'+1}\geq\dots\geq k_t$ and $k_1+k_{s+1}\leq n <k_{s-1}+k_s$.
Suppose that $A_1, B_1, C_1\in \mathcal{R}_1$ with $A_1\prec B_1\prec C_1$. Then
for any $i\in [s]$,
$
\alpha_i(A_1, C_1)=\alpha_i(A_1, B_1)+\alpha_i(B_1, C_1)$, and
$\gamma(A_1, C_1)=\gamma(A_1, B_1)+\gamma(B_1, C_1)$,
 $\delta(A_1, C_1)=\delta(A_1, B_1)+\delta(B_1, C_1)$.
\end{remark}

By the definition of parity,  we have the following remark.
\begin{remark}\label{NEW12}
Let $s, s', t$ be positive integers with $s\geq s'\geq 1$, $t\geq s+1$, $k_1=\dots=k_{s'}> k_{s'+1}\geq\dots\geq k_t$ and $k_1+k_{s+1}\leq n <k_{s-1}+k_s$. Then
for any $R\in \mathcal{R}_1$ and $i\in [2, s]$, $R$ has the $k_i$-parity if and only if $|R\setminus R^{\rm t}|\leq k_i$.
\end{remark}

\begin{claim}\label{claim1}
Let $s, s', t$ be integers with $s\geq s'\geq 1$, $t\geq s+1$, $k_1=\dots=k_{s'}> k_{s'+1}\geq\dots\geq k_t$ and $k_1+k_{s+1}\leq n <k_{s-1}+k_s$.
Let  $R_1, R'_1\in \mathcal{R}_1$ with $R_1<R'_1$ and $R=R'_1\setminus {R'_1}^{\rm t}$.
Then for each $i\in [s]$, we have
\[
\alpha_i(R_1, R'_1)={\ell(R'_1)\choose k_i-|R|}.
\]
In particular, $\alpha_i(R_1, R'_1)=0$ if and only if $\ell(R'_1)< k_1-k_i$. Furthermore, if $\ell(R'_1)=0$, then $\gamma(R_1, R'_1)=s'$.
\end{claim}

\begin{proof}
Let  $i\in [s]$.
Clearly, if $i\in [s']$, then $\alpha_i(R_1, R'_1)=1={\ell(R'_1)\choose \ell(R'_1)}={\ell(R'_1)\choose k_i-|R|}$.
We next consider  $i\in[s'+1, s]$. In this case, $k_i<k_1$.
Let $R_i$ and $R'_i$ be the corresponding $k_i$-sets of $R_1$ and $R'_1$, respectively.  Then $R_i\prec R'_i$, moreover $\alpha_i(R_1, R'_1) \geq 0$ and  $\alpha_i(R_1, R'_1)=0$ if and only if $R_i=R'_i$. By Fact \ref{correspondingk}, $R_i=R'_i$ if and only if $R'_1$ does not have $k_i$-parity, i.e., $\ell(R'_1)< k_1-k_i$ in view of Remark \ref{NEW12}.
In particular, if $\ell(R'_1)=0$, then $\alpha_i(R_1, R'_1)=0={\ell(R'_1)\choose k_i-|R|}$. Therefore, $\gamma(R_1, R'_1)=s'$,
as required. We complete the proof of the case that $R'_1$ does not have $k_i$-parity.

Next, we assume that $R'_i$ is the $k_i$-parirty of $R'_1$. So $\ell(R'_1)\geq k_1-k_i$. In this case,
since $R_1<R'_1$, $\ell(R_1)=\ell(R'_1)-1$. We first assume that $R_1$ does not have $k_i$-parity. Then $\ell(R_1)\leq k_1-k_i-1$.
Since $\ell(R_1)=\ell(R'_1)-1$,
$\ell(R'_1)= k_1-k_i$. Note that $\ell(R'_1)=k_1-|R|$. So $k_i=|R|$ and $R'_i=R$. Note that $R_1<R'_1$,  $\ell(R_1)=\ell(R'_1)-1$ and $R_1$ does not have $k_i$-parity. By the definition of corresponding $k_i$-set, we obtain $R_i<R'_i$. So
$\alpha_i(R_1, R'_1)=1={\ell(R'_1)\choose k_i-|R|}$, as required.
At last, we  assume that $R_i$ is the $k_i$-parirty of $R_1$. So $\ell(R_1)\geq 1$.
Let $k_1-k_i=k$.
In this case we have
\begin{align*}
R'_1&=R\cup [n-\ell(R'_1)+1, n];\\
R_1&=R\cup \{n-\ell(R'_1)\} \cup [n-\ell(R'_1)+2, n];\\
R_i&=R\cup \{n-\ell(R'_1)\}\cup [n-\ell(R'_1)+2+k, n];\\
R'_i&=R\cup [n-\ell(R'_1)+1+k, n]=R\cup [n-k_i+|R|+1, n].
\end{align*}
Let $A$ be the $k_i$-set such that $R_i<A$.
Then
$$A=R\cup \{n-\ell(R'_1)+1, \dots, n-\ell(R'_1)+k_i-|R|\}.$$
Denote $\mathcal{F}=\{F\in {[n]\choose k_i}: A\prec F\prec R'_i\}.$
Then
$
\alpha_i(R_1, R'_1)=|\mathcal{F}|={n-(n-\ell(R'_1))\choose k_i-|R|}={\ell(R'_1)\choose k_i-|R|},
$
as required.
\end{proof}

\begin{claim}\label{0000}
Let $s, s', t$ be integers with $s\geq s'\geq 1$, $t\geq s+1$, $k_1=\dots=k_{s'}> k_{s'+1}\geq\dots\geq k_t$ and $k_1+k_{s+1}\leq n <k_{s-1}+k_s$.
Let  $F_1, F'_1, G_1, G'_1\in \mathcal{R}_1$ be such that $F_1<F'_1, G_1<G'_1$ and  $\max F'_1=\max G'_1$.
Then

(i) $\delta(F_1, F'_1)=\delta(G_1, G'_1)$;

(ii) if $\ell(F'_1)\leq \ell(G'_1)$, then $\gamma(F_1, F'_1)\leq \gamma(G_1, G'_1)$, and if $\ell(F'_1)= \ell(G'_1)$, then $\gamma(F_1, F'_1)=\gamma(G_1, G'_1)$.
\end{claim}

\begin{proof}
 Since $F_1<F'_1, G_1<G'_1$ and $\max F'_1=\max G'_1$, it follows from Claim \ref{0000-} and Remark \ref{remark4.14} that
$$\delta(F_1, F'_1)=\delta(G_1, G'_1),$$ and for each $i\in [s']$,
  \begin{equation}\label{NEW15}
\alpha_i(F_1, F'_1)=\alpha_i(G_1, G'_1).
\end{equation}
Note that
\begin{equation}\label{NEW22}
|F'_1\setminus {F'_1}^{\rm t}|+\ell(F'_1)=|G'_1\setminus {G'_1}^{\rm t}|+\ell(G'_1)=k_1.
\end{equation}
If $\ell(F'_1)\leq \ell(G'_1)$, then for each $i\in [s'+1, s]$, it follows from Claim \ref{claim1} and (\ref{NEW22}) that
\begin{equation}\label{NEW21}
\alpha_i(F_1, F'_1)={\ell(F'_1)\choose k_i-|F'_1\setminus {F'_1}^{\rm t}|}\leq{\ell(G'_1)\choose k_i-|G'_1\setminus {G'_1}^{\rm t}|}=\alpha_i(G_1, G'_1).
\end{equation}
Combing  (\ref{NEW15}) and (\ref{NEW21}), we obtian
\begin{equation}\label{NEW23}
\gamma(F_1, F'_1)=\sum_{i=1}^s\alpha_i(F_1, F'_1)\leq \sum_{i=1}^s\alpha_i(G_1, G'_1)=\gamma(G_1, G'_1),
\end{equation}
as required.
Moreover, if $\ell(F'_1)= \ell(G'_1)$, then 
equality holds in (\ref{NEW21}), therefore, equality holds in (\ref{NEW23}) as well. 
 This completes the proof of  Claim \ref{0000}.
\end{proof}

\begin{claim}\label{equal}
Let $s, s', t$ be integers with $s>s'\geq 1$, $t\geq s+1$, $k_1=\dots=k_{s'}> k_{s'+1}\geq\dots\geq k_t$ and $k_1+k_{s+1}\leq n <k_{s-1}+k_s$.
Let  $d\ge 1$ and $A_1, B_1, C_1, D_1\in \mathcal{R}_1$.
Suppose that $A_1, B_1$ are $d$-sequential and  $C_1, D_1$ are $d$-sequential with $\max A_1=\max C_1$ and $\max B_1=\max D_1$. Then $\gamma(A_1, B_1)=\gamma(C_1, D_1)$ and
$\delta(A_1, B_1)=\delta(C_1, D_1)$.
In particular,
if $A_1\overset{d}{\longrightarrow} B_1$, $C_1\overset{d}{\longrightarrow} D_1$ and $\max A_1=\max C_1$, then $\gamma(A_1, B_1)=\gamma(C_1, D_1)$ and
$\delta(A_1, B_1)=\delta(C_1, D_1)$.
\end{claim}

\begin{proof}
 By the definitions of $A_1, B_1, C_1, D_1$, from Lemma \ref{clm27-} and Remark \ref{remark4.14}, we have
 $\delta(A_1, B_1)=\delta(C_1, D_1)$ and $\alpha_i(A_1, B_1)=\alpha_i(C_1, D_1)$ holds for each $i\in [s']$.  Next, we only need to show that for each $i\in [s'+1, s]$, $\alpha_i(A_1, B_1)=\alpha_i(C_1, D_1)$. Let $i\in [s'+1, s]$.
Denote
\begin{align*}
\mathcal{A}&=\Big\{R\in {[n]\choose k_1}: A_1\prec R\prec B_1\Big\},\\
\mathcal{B}&=\Big\{T\in {[n]\choose k_1}: C_1\prec T\prec D_1\Big\}.
\end{align*}
Since $\alpha_1(A_1, B_1)=\alpha_1(C_1, D_1)$, $|\mathcal{A}|=|\mathcal{B}|=:h$.
Let $\mathcal{A}=\{R_1, R_2, \dots, R_h\}$ and $\mathcal{B}=\{T_1, T_2, \dots, T_h\}$, where $R_1 \prec R_2 \prec \dots \prec R_h$ and $T_1 \prec T_2 \prec \dots \prec T_h$.
For any $j\in [h]$, we have $\ell(R_j)=\ell(T_j)$ and $|R_j\setminus R_j^{\rm t}|=|T_j\setminus T_j^{\rm t}|$. Thus, by Claim \ref{claim1}, for any $j\in [h-1]$, $\alpha_i(R_j, R_{j+1})=\alpha_i(T_j, T_{j+1})$. Furthermore, by Remark \ref{add}, we conclude that
$$\alpha_i(A_1, B_1)=\sum_{j\in [h-1]}\alpha_i(R_j, R_{j+1})=\sum_{j\in [h-1]}\alpha_i(T_j, T_{j+1})=\alpha_i(C_1, D_1),$$
as required.
\end{proof}

Now, we are ready to prove Lemma \ref{claim4}.
\begin{proof}[Proof of Lemma \ref{claim4}]
From the definitions of $A_1, B_1$ and $C_1$, we have $\ell(C_1)\geq \ell(B_1)>\ell(A_1)\geq 0$. Thus $\max B_1=\max C_1=n$.
If $\ell(C_1)> \ell(B_1)$, then $A_1<B_1<C_1$, then by Claim \ref{0000}, we have
$\delta(A_1, B_1)=\delta(B_1, C_1)$ and $\gamma(A_1, B_1)\leq \gamma(B_1, C_1)$, we are done.
Next we may assume that $\ell(C_1)= \ell(B_1)$.  Let $A'_1$ and $B'_1$ be the $k_1$-sets such that $A_1<A'_1$ and $B_1<B'_1$.
Then $\max A'_1=\max B'_1$, $\max B_1=\max C_1=n$.
Note that $\ell(C_1)= \ell(B_1)$ implies that $A_1<B_1<C_1$ does not happen. So $A'_1\ne B_1$ and $B'_1\ne C_1$,
 and $A'_1$, $B_1$ are $(\ell(A_1)+1)$-sequential, $B'_1$, $C_1$ are $(\ell(A_1)+1)$-sequential.
By Claim \ref{equal},
$\gamma(A'_1, B_1)=\gamma(B'_1, C_1)$ and $\delta(A'_1, B_1)=\delta(B'_1, C_1)$. Since $A_1<A'_1$ and $B_1<B'_1$ and $\max A'_1=\max B'_1$, by Claim \ref{0000},
$\gamma(A_1, A'_1)\leq \gamma(B_1, B'_1)$ and $\delta(A_1, A'_1)=\delta(B_1, B'_1)$. In view of Remark \ref{add}, we have
\begin{align*}
&\gamma(A_1, B_1)=\gamma(A_1, A'_1)+\gamma(A'_1, B_1)\leq \gamma(B_1, B'_1)+\gamma(B'_1, C_1)=\gamma(B_1, C_1),\\
&\delta(A_1, B_1)=\delta(A_1, A'_1)+\delta(A'_1, B_1)= \delta(B_1, B'_1)+\delta(B'_1, C_1)=\delta(B_1, C_1),
\end{align*}
as required.
\end{proof}

\section{Acknowledgements}
We would like to thank Andrey Kupavskii for a number of helpful comments and suggestions.


\end{document}